%% file: main.tex
\documentclass[11pt]{amsart}

\usepackage[utf8]{inputenc}
\usepackage[T1]{fontenc}

\usepackage[english]{babel}
\usepackage[shortlabels]{enumitem}
\usepackage{mathtools,amssymb}
\usepackage{hyperref}
\usepackage{cleveref}
\usepackage{xcolor,graphicx}
\usepackage{geometry}

\usepackage[center,font={small}]{caption}

\crefname{lstlisting}{listing}{listings}
\Crefname{lstlisting}{Listing}{Listings}

\usepackage{tikz}
\usetikzlibrary{calc,positioning,decorations.markings,arrows.meta,backgrounds,patterns,shadows}

\newif\ifbw

\input{tikz/colors}

\usepackage{listings}
\lstdefinelanguage{GAP}{%
  morekeywords={%
    and,break,continue,do,elif,else,end,false,fi,for,%
    function,if,in,local,mod,not,od,or,quit,rec,repeat,return,%
    then,true,until,while%
  },%
  sensitive,%
  morecomment=[l]\#,%
  morestring=[b]",%
  morestring=[b]',%
}[keywords,comments,strings]
\colorlet{myblue}{blue!70!black}
\usepackage[all]{xy}

\newtheorem{theorem}{Theorem}[section]
\newtheorem{lemma}[theorem]{Lemma}
\newtheorem{proposition}[theorem]{Proposition}
\newtheorem{corollary}[theorem]{Corollary}
\newtheorem{question}[theorem]{Question}

\newtheorem{thmA}{Theorem}

\theoremstyle{definition}
\newtheorem{definition}[theorem]{Definition}

\theoremstyle{remark}
\newtheorem{rem}[theorem]{Remark}
\newtheorem{example}[theorem]{Example}

\newcommand{\Z}{\mathbb{Z}}
\newcommand{\N}{\mathbb{Z}_+}
\newcommand{\NN}{\mathbb{Z}_{\ge0}}

\newcommand{\R}{\mathbb{R}}

\newcommand{\HH}{\mathcal{H}}
\newcommand{\OO}{\mathcal{O}}

\newcommand{\T}{\mathbb{T}}

\newcommand{\SL}{\operatorname{SL}}

\newcommand{\id}{\operatorname{id}}

\newcommand{\aut}{\operatorname{Aut}}

\newcommand{\mult}{\operatorname{mult}}
\newcommand{\ord}{\operatorname{ord}}
\usepackage{xspace}
\newcommand{\property}{property~\textnormal{(}$\mathcal{C}$\textnormal{)}\xspace}
\newcommand\inv{^{-1}}

\subjclass[2010]{32G15, 14H30, 57M10, 20D15}

\author{Johannes Flake}
\address{Algebra and Representation Theory,
RWTH Aachen University,
Pontdriesch 10-16,
52062 Aachen, Germany}
\email{flake@art.rwth-aachen.de}

\author{Andrea Thevis}
\address{Department of Mathematics and Computer Science,
Saarland University,
66123 Saarbr\"ucken, Germany}
\email{thevis@math.uni-sb.de}

\begin{document} 

\title{Strata of \texorpdfstring{$p$}{p}-Origamis}

\begin{abstract} Given a two-generated group of prime-power order, we investigate the singularities of origamis whose deck group acts transitively and is isomorphic to the given group. Geometric and group-theoretic ideas are used to classify the possible strata, depending on the prime-power order. We then show that for many interesting known families of two-generated groups of prime-power order, including all regular, or powerful ones, or those of maximal class, each group admits only one possible stratum. However, we also construct examples of two-generated groups of prime-power order which do not determine a unique stratum.
\end{abstract}

\maketitle

\tableofcontents

\section{Introduction}
In this article, we study a certain class of translation surfaces. The theory of translation surfaces has been a very active research area over the past 40 years with connections to mathematical billiards, dynamical systems, and Teich\-m\"uller theory. These applications are for example discussed in the survey articles \cite{HubertSchmidt}, \cite{wright2014}, and \cite{Zorich06}.

The strata $\HH_g(k_1\times a_1,\dots,k_m\times a_m)$, $a_i,k_i\in \N$, stratify the space of translation surfaces of genus $g$. Here the \emph{stratum} $\HH_g(k_1\times a_1,\dots,k_m\times a_m)$ consists of all translation surfaces of genus $g$ with $k_i$ singularities of multiplicity $a_i+1$ for $1\le i\le m$. The study of strata has played an important role ever since the fundamental work of Masur, Smillie, and Veech in the 1980s (see \cite{Masur,Veech89,KMS86}). A first essential result was achieved by Kontsevich and Zorich with the classification of the connected components of strata in \cite{Kontsevich_2003}. A natural $\SL(2,\R)$-action on each stratum is crucial in the study of translation surfaces. Eskin and Mirzakhani described in their groundbreaking work \cite{eskin2013invariant} the closure of the $\SL(2,\R)$-invariant subspaces. 

In our article, we are interested in certain translation surfaces which we call $p$-origamis. Origamis, also known as square-tiled surfaces, are finite torus covers and form a particularly interesting class of translation surfaces. On the one hand, the set of origamis is dense in each stratum. On the other hand, each origami defines a Teichm\"uller curve, i.e., an algebraic curve induced by the $\SL(2,\R)$-orbit of certain translation surfaces. Moreover, each algebraic curve over $\overline{\mathbb{Q}}$ is birational to a Teichm\"uller curve arising from the $\SL(2,\R)$-orbit of an origami (see \cite{Ellenberg}). In general, the classification of $\SL(2,\R)$-orbits of origamis is an unsolved problem. However, in the stratum $\HH(2)$, the possible orbits of origamis have been classified in \cite{HubertLelievre06} and \cite{McM}. Furthermore, counting problems of origamis are related to the study of Masur--Veech volumes of strata (see e.g. \cite{delecroix2016, delecroix2019, aggarwal2019, GM20}). We refer the interested reader to \cite{schmithuesen2004algorithm}, \cite{schmithuesen2007origamis}, and \cite{Zmiaikou} as introductions to origamis. 

Origamis arising as normal torus covers are called regular or \emph{normal origamis}. In \cite{MYZ}, the homology groups of normal origamis are studied, and results on the Lyapunov exponents of the Kontsevich--Zorich cocycle (which capture certain dynamical properties of a surface) are deduced.

Each normal origami is determined by its deck transformation group $G$ and a particularly chosen pair of generators $(x,y)$ of $G$. This description allows us to examine such origamis using group theory. One question which arises naturally asks in which strata normal origamis occur. It turns out that, if $G$ has order $d$ and the commutator of $x$ and $y$ has order $a$, then the stratum of the corresponding origami is $\mathcal{H}(\frac da \times (a-1))$ (see \Cref{rem: connection stratum - 2-gen gps}). 

As special cases of normal origamis, we consider those whose deck group is a finite $p$-group, that is, a group of prime-power order. We call such origamis \emph{$p$-origamis}.
A motivating example for studying them is the well-known origami \emph{Eierlegende Wollmilchsau} (see \Cref{eierlegende wollmilchsau} and \cite{HerSchmit}). It is one out of two translation surfaces whose $\SL(2,\R)$-orbit induces not only a Teichm\"uller curve, but also a Shimura curve in the moduli space of abelian varieties. As a consequence, its Teichm\"uller curve has extraordinary dynamical behavior (see \cite{aulicino2019} and \cite{Moeller}).

We prove a precise characterization of all strata possible for $p$-origamis. As for many questions in the theory of $p$-groups, the situation is fundamentally different for the even prime $2$ and for all other odd primes. 

\begin{thmA}[\Cref{strata 2-origamis}, \Cref{strata p-origamis for p>2}] \label{thm-A}\label{intro-strata 2-origamis}\label{intro-strata p-origamis for p>2}
	Let $n\in \NN$. Then any $p$-origami of degree $p^n$ has either no singularity and genus $1$, or lies in one of the following strata:
	\begin{enumerate}[\textbullet]
		\item $\mathcal{H}\left( 2^{n-k} \times \left(2^k-1\right)\right),$ for $1\le k\le n-2$, if $p=2$,
		\item $\mathcal{H}\left( {p^{n-k}}\times \left( p^k -1\right) \right) ,$ for $1\le k <\frac{n}{2}$, if $p>2$.
	\end{enumerate}	
	Moreover, all of these strata occur.
\end{thmA}

Consider a fixed abelian $2$-generated $p$-group. Then the commutator of any pair of generators is trivial, and hence, so is the stratum of any $p$-origami with the fixed group as its deck group. Two observations can be generalized from this simple example: First, we prove that the possible strata of a $p$-origami only depend on the isoclinism class of its ($p$-)group of deck transformations. Since all abelian groups are isoclinic to the trivial group, this indeed generalized our toy example, while it also implies, for instance, that the dihedral, the semidihedral, and the quaternion group of order $2^k$ for some $k\geq 1$ admit the same possible strata (see \Cref{rem-stratum-maximal-class-2gps}, where we also compute the exact stratum).

Second, we show that far beyond the abelian case, the deck group determines a unique stratum -- one which is independent of the choice of generators $x,y$ -- in many more situations.

\begin{thmA}[\Cref{thm-deck-group-determines-stratum}] \label{thm-B}\label{thm-groups-B}
    Many deck groups of prime-power order admit only one possible stratum for their $p$-origamis, including all $p$-groups $G$ which are regular, of maximal class, powerful, or those whose commutator subgroup $G'$ is regular, powerful or order-closed. This includes all $p$-groups of order up to $p^{p+2}$ or of nilpotency class up to $p$. 
\end{thmA}

Our results on strata of $p$-origamis are obtained by an array of group-theoretic methods. Given a $2$-generated $p$-group $G$, the possible strata of all $p$-origamis with deck group $G$ depend on the possible orders of the commutators $[x,y]$ as $x,y$ vary over all pairs of generators. We derive various results on the possible exponent of the commutator subgroup, which contains the commutator $[x,y]$ and, in fact, in the case of $2$-generated groups, is generated by the set of its conjugates. We then show these exponents can always be realized as commutator orders of pairs of generators. This culminates in a complete characterization of the possible commutator orders $[x,y]$ for each prime-power group order, which forms the group-theoretic analogue of \Cref{thm-A}:
\begin{thmA}[\Cref{prop. exp(G')<n-1}, \Cref{constr of 2-gps}, \Cref{thm bound for exp(G')}, \Cref{group theory: strata p-origamis for p>2}]~
\begin{enumerate} \label{intro-prop. exp(G')<n-1}
    \item For any finite $2$-group $G$, $\exp(G')=1$ if $|G|\leq 2$, or else $\exp(G')\le \tfrac{|G|}{4}$.
    \item \label{intro-constr of 2-gps}
    For all integers $n\geq 2$ and $0\leq k\le n-2$, there exists a $2$-generated $2$-group $G$ of order $2^n$ with generators $x,y$ such that
    $$\ord([x,y])=\exp(G')=2^k.$$
    \item For any non-trivial finite $p$-group $G$ with odd $p$, $\exp(G')^2<|G|$.
	\item For any odd prime $p$ and any $n,k\in \NN$ with $k< \frac{n}{2}$, there exists a $2$-generated $p$-group $G$ of order $p^n$ with generators $x,y$ such that 
	$$\ord([x,y])=\exp(G')=p^k.$$
\end{enumerate}
\end{thmA}

To investigate the question, in which cases the deck group already determines the stratum of a $p$-origami (as partially answered in \Cref{thm-B}), we similarly translate this into a group-theoretic problem. In group-theoretic language, this phenomenon corresponds to the property of a  $p$-group, that the commutator order is a fixed number for all pairs of generators. We call this \property, and we observe that many, but not all $p$-groups have \property. We solidify this observation by proving \property depending on some other known common properties of $p$-groups. The proven implications can be summarized with the following diagram (\Cref{thm-diagram}), they form the group-theoretic basis of \Cref{thm-B}:
\newcommand{\diagram}{
\[\small\begin{tikzpicture}[remember picture, node distance=8mm and 16mm,>={Latex[length=2mm]},shorten >=1mm,align=center]

\node[](a){$G$ regular};
\node[below=of a](b){$G$ order-closed};
\node[below=of b](c){$G$ power-closed};

\node[right=of a](a1){$G'$ regular};
\node[below=of a1](b1){$G'$ order-closed};
\node[below=of b1](c1){$G'$ power-closed};
\node[right=10mm of b1](b2){$G'$ weakly\\order-closed};
\node[right=10mm of c1](c2){$G'$ weakly\\power-closed};

\node[above=of a1](mc){$G$ maximal class};

\node[right=50mm of mc](pf){$G$ powerful};
\node[below=of pf](pf1){$G'$ powerful};

\node[below=12mm of pf1,fill=black!10](x){$G$ has \property};

\draw[->] (a)--(b);
\draw[->] (b)--(c);
\draw[->] (a1)--(b1);
\draw[->] (b1)--(c1);
\draw[->] (a)--(a1);
\draw[->] (b)--(b1);
\draw[->] (c)--(c1);
\draw[->] (b1)--(b2);
\draw[->] (c1)--(c2);

\draw[->] (mc)--(a1);
\draw[->] (pf)--(a1);
\draw[->] (pf)--(pf1);

\draw[->] (b2)--(x);

\draw[->] (pf1)--(x);

\draw[->,decoration={markings, mark=at position 0.5 with
    { \draw[very thick,-] ++ (-5pt,-5pt) -- ( 5pt, 5pt); }
    },postaction={decorate}] (c2)--(x);
\begin{scope}[on background layer]
\draw[draw=black!10,ultra thick,transform canvas={yshift=-0.3cm}] (b.south west) -- (b.south west -| x.east);
\end{scope}
\end{tikzpicture}\]
}
\diagram
We also construct $p$-groups which do not have \property. Such groups have at least order $p^{p+3}$: we construct the examples as subgroups of the Sylow $p$-subgroup of the symmetric group $S_{p^4}$. 

It is an open question which power-closed groups, or which groups with a power-closed commutator subgroup, have \property. As it seems hard to find (small) $p$-groups without \property, their properties or even their classification might be an intriguing problem. Choices of commutators with different orders in such groups will yield normal origamis with isomorphic deck groups lying in different strata.

As an outlook, we consider surfaces that arise as infinite normal torus covers. These surfaces form a special class of infinite translation surfaces. We focus on surfaces with dense subgroups of profinite groups as deck groups, examples include origamis with the infinite dihedral group as deck group. After generalizing the definition of \property to profinite and pro-$p$ groups, we transfer results from \Cref{section: gp thy iso deck groups} on finite $p$-groups to these new situations. As in the finite case, the group-theoretic results have a geometric interpretation concerning the singularities of infinite translation surfaces. 

\textbf{Structure of the paper.} The article is organized as follows. In \Cref{section: geometric motivation}, we introduce essential definitions for the geometric point of view and explain the geometric motivation of the questions studied in the subsequent sections. 

Developing the group-theoretic results is subject of \Cref{section: results on p-groups}. In \Cref{section: gp thy for strata results}, we prove \Cref{intro-prop. exp(G')<n-1}. For this, we consider the prime $2$ separately because the results differ from those for odd primes. \Cref{section: gp thy iso deck groups} is concerned with the group-theoretic analogue of \Cref{thm-B}. More precisely, we introduce a group-theoretic property called \property and prove the above diagram of implications. We further construct examples of $p$-groups that do not have \property. 

The goal of \Cref{section: results on p-oris} is to translate the results for $p$-groups in \Cref{section: results on p-groups} into the language of $p$-origamis. In \Cref{subsection: strata of p-origamis}, we answer the question in which strata $p$-origamis occur by proving \Cref{thm-A}. The question whether the isomorphism class of the deck group determines the stratum of an origami is addressed in \Cref{section: isomorphic deck transformation groups}. \Cref{thm-B} is proven there. 

In \Cref{section infinite origamis}, we generalize some results from \Cref{section: results on p-groups} and \Cref{section: results on p-oris} to profinite and pro-$p$ groups, and infinite translation surfaces arising as normal torus covers, respectively.

\textbf{Acknowledgments.} We would like to thank Gabriela Weitze-Schmith\"usen, Alice Nie\-meyer, and Charles Leedham-Green for fruitful discussions involving origamis and $p$-groups as well as Gabriela Weitze-Schmith\"usen for careful proofreading. We are also grateful to Stephen Glasby, Charles Leedham-Green, and Wilhelm Plesken for helpful comments on an earlier version of this paper and to Dominik Bernhardt for discussions about \texttt{GAP} computations. The second author is thankful for the support of the National Science Foundation under Grant No.~1440140 during a research stay at the Mathematical Sciences Research Institute in Berkeley in the fall semester of 2019. Moreover, the second author would like to thank the German Academic Scholarship Foundation for supporting her doctoral studies.

Both authors are funded by the SFB-TRR 195 `Symbolic Tools in Mathematics and their Application’. This article is part of Project I.8.

\section{Connections between geometry and deck groups}\label{section: geometric motivation}

We begin by recalling the construction of translation surfaces and origamis. A translation surface is constructed from finitely many polygons in the Euclidean plane, where pairs of parallel edges (of the polygons) are identified by translations. For an origami, take finitely many copies of the unit square. Glue them along their edges via translations such that each left edge is glued to exactly one right edge and each upper edge to exactly one lower edge. The resulting surface is a translation surface. We require that the surface is connected. Otherwise one studies the connected components separately. Such a translation surface is called \textbf{origami} or \textbf{square-tiled surface}. An origami $\OO$ naturally defines a covering $\OO\to \T$ of the torus $\T$ ramified at most over one point denoted by $\infty$. The number of glued squares is the degree of the covering.

In this section, \cite{Zorich06} and \cite{HubertSchmidt} are used as main references for well-known facts about translation surfaces. Further, we refer the interested reader to \cite{Forster} for background knowledge about coverings. 
\begin{example}
	The following origami is a cover of the torus of degree 4. Edges with the same labels are identified.
	\begin{center}
    	\captionsetup{type=figure}
		\includegraphics[scale = 0.2]{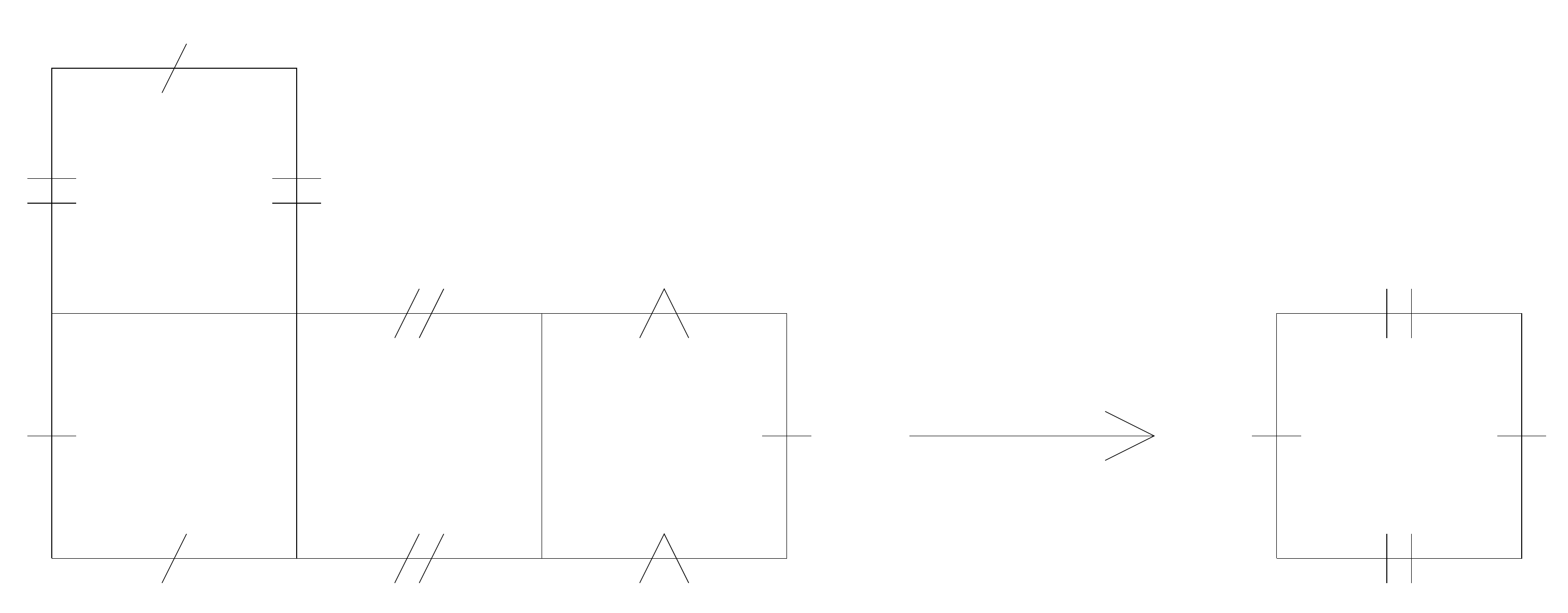}
		\caption{An example of an origami defining a torus cover which is not normal.}
	\end{center}
\end{example}

The concept of monodromy maps is essential for relating the stratum of a normal origami with group theoretic properties of its deck group. Let $c:\OO\to \T$ be the covering induced by an origami of degree $d$. Recall that its \textbf{deck transformation group} consists of all homeomorphisms $f:\OO\to\OO$ such that $c\circ f=c$. Consider the corresponding unramified cover of the punctured torus $c^*:\OO^*\to \T^*$, where $\OO^*=\OO\setminus c^{-1}(\infty)$ and $\T^*=\T\setminus\{\infty\}$. Recall that the fundamental group $\pi_1(\T^*)$ is the free group $F_2$ on two generators. 
We choose a base point $q$ on $\T^*$ and label the preimages of $q$ (under $c$) by $q_1,\dots,q_d$. Denote the simple closed horizontal and vertical curve passing through $q$ by $a$ and $b$, respectively. These two curves generate the fundamental group $F_2$. Then the \textbf{monodromy map} $m:F_2\to S_d,~w\mapsto \sigma_w$ is an anti-homomorphism defined as follows: A word $w\in F_2$ describes a path $f$ on $\T^*$. For $1\le i\le d$, one obtains a lifted path $\tilde{f}_i$ of $f$ on $\OO^*$ starting at $q_i$. Set $\sigma_w(i)\vcentcolon=j$, if the path $\tilde{f}_i$ ends at $q_j$. This defines a permutation $\sigma_w\in S_d$.

A \textbf{normal} (or \textbf{regular}) origami is an origami which is a normal cover of the torus, i.e., the deck transformation group acts transitively on the squares. Let $c:\OO\to \T$ be the covering induced by a normal origami. Then the degree of the cover is the order of the deck transformation group. We obtain a natural bijection between the squares of the tiling and the deck transformations by labeling a fixed square $S$ with the identity. A square $S'$ is now labeled with the unique deck transformation sending the square $S$ to $S'$. Using this bijection, the monodromy action is described by a map $m_G: F_2\to G, w\mapsto g_w$ such that the group homomorphism $G\to G, h\mapsto g_w\cdot h$ induces the permutation $m(w)$ on the squares of the tiling. Note that the deck group acts from the left on the origami $\OO$.

\begin{lemma}\label[lemma]{p-ori = 2-gen set}
	The following holds:
		\begin{enumerate}[(i)]
			\item A finite $2$-generator group $G$ together with an (ordered) pair of generators determines a normal origami with deck transformation group $G$.
			\item A normal origami is uniquely determined by its deck transformation group $G$ and the two deck transformations $m_G(a)$ and $m_G(b)$. 
	\end{enumerate}
\end{lemma}

\begin{proof}
	(i) Given a $2$-generator group $G$ of order $d$ together with generators $x,y$, we can construct a normal origami of degree $d$ as follows. Take $d$ squares labeled by the group elements. The right and upper neighbor of a square labeled by $g$ in $G$ is the one with label $gx$ and $gy$, respectively. This construction data defines an origami of degree $d$ with deck transformation group $G$. By definition, $G$ acts transitively on the squares and thus the cover is normal.
	
	(ii) Given a normal origami $\OO\to \T$ with deck transformation group $G$. Consider the deck transformations $x\vcentcolon=m_G(a)$ and $y\vcentcolon=m_G(b)$. These deck transformations represent passing from the square labeled by the identity element of $G$ to its right and upper neighbor, respectively. The deck group $G$ is generated by $x$ and $y$. The procedure described in (i) reconstructs the origami $\OO$ from the data $(G,x,y)$.
\end{proof}

Recall that the monodromy map is an anti-homomorphism and induces a left action of the deck group. In the construction described in the proof of \Cref{p-ori = 2-gen set}, we multiply the generators from the right to obtain compatibility with this left action.

\begin{example}\label[example]{eierlegende wollmilchsau}
	We consider the quaternion group $Q_8\vcentcolon= \left\langle  i, j, k~|~i^2=j^2=k^2=ijk \right\rangle $ with $(i,j)$ as the pair of generators. $Q_8$ can be viewed as a group of units in the quaternion division algebra, where $ijk=-1$. Following the construction described in \Cref{p-ori = 2-gen set} i), we obtain a $2$-origami $\mathcal{W}$ with deck transformation group $Q_8$. In the picture below, the corners of the squares with the same color are identified, i.e., the cone angles at these points are larger than $2\pi$.
	\begin{figure}[ht]
	\centering
	    \input{tikz/wollmilchsau}
	    \caption{The Eierlegende Wollmilchsau is a $2$-origami of degree $8$ with $4$ singularities of cone angle $4\pi$.}
	    \label[figure]{fig:my_label}
	\end{figure}
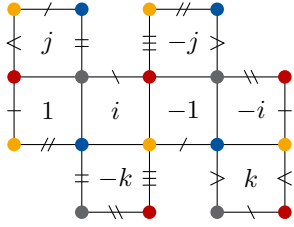
	This origami is called \textbf{Eierlegende Wollmilchsau}. It is a well-known and extensively studied example (see~\cite{HerSchmit}).
\end{example}

We call normal origamis $c_1:\OO_1\to \T$ and $c_2:\OO_2\to \T$ equal if there exists a homeomorphism $\alpha:\OO_1\to \OO_2$ such that $c_1=c_2\circ \alpha$. It is natural to ask when different pairs of generators for a given group describe the same origami. The following lemma answers this question. 

\begin{lemma}\label[lemma]{origamis are equal}
	Let $\mathcal{O}_1$ and $\mathcal{O}_2$ be normal origamis with deck transformation group $G$ defined by pairs of generators $(x_1,y_1)$ and $(x_2,y_2)$, respectively. Let  $m_i: F_2\to G$ with $m_i(a)=x_i$ and $m_i(b)=y_i$ denote the monodromy maps of $\mathcal{O}_i$ for $i=1,2$. Then the following are equivalent
		\begin{itemize}
			\item[(i)] the origamis $\mathcal{O}_1$ and $\mathcal{O}_2$ are equal,
			\item[(ii)] the kernels of the monodromy maps $m_1$ and $m_2$ are equal,
			\item[(iii)] there exists a group automorphism $\varphi:G\to G$ such that $(\varphi(x_1),\varphi(y_1))=(x_2,y_2)$.
	\end{itemize}
\end{lemma}

\begin{proof}
	We begin by showing that (i) implies (ii). Let $\alpha:\OO_1\to \OO_2$ be a homeomorphism such that $c_1=c_2\circ \alpha$. Recall that an element $w\in F_2$ defines a path $f$ on $\T^*$ passing through a chosen base point $q_1$. If $w$ lies in the kernel of $m_1$, each lift $\tilde{f}_i$ on $\OO_1^*$ of $f$ starts and ends at the same preimage of $q_1$. Further, each of the paths $\tilde{f}_i$ on $\OO_1^*$ induces a path $\alpha(\tilde{f}_i)$ which also has the same start and end point. This implies that $w$ lies in the kernel of $m_2$. Hence we obtain the inclusion $\ker(m_1)\subseteq \ker(m_2)$. Using the inverse $\alpha^{-1}$, a similar argument shows the inclusion $\ker(m_2)\subseteq \ker(m_1)$.
	
	Now we show that (ii) implies (iii). Suppose that the kernels of $m_1$ and $m_2$ are equal. We want to define a group isomorphism $\varphi: G\to G$ such that the following diagram commutes
	\begin{align*}
	\xymatrix{
		F_2 \ar[d]_-{\id} \ar[r]^-{m_{1}} & G\ar@{-->}[d]^-{\varphi} \\
		F_2 \ar[r]^-{m_2} & G.
	}
	\end{align*}
	Let $K$ denote the kernel $\ker(m_1)$. Since the kernels of $m_1$ and $m_2$ coincide, we obtain isomorphisms $\overline{m}_i:F_2/K \to G$ with $\overline{m}_i(\overline{a})=x_i$ and $\overline{m}_i(\overline{b})=y_i$ for $1\le i\le 2$. Define $\varphi$ as the composition $\overline{m}_2\circ(\overline{m}_1)^{-1}$. Then $\varphi$ is an isomorphism such that $(\varphi(x_1),\varphi(y_1))=(x_2,y_2)$.
	
	Finally, we assume statement (iii). Let $\varphi:G\to G$ be an automorphism such that $(\varphi(x_1),\varphi(y_1))=(x_2,y_2)$. Sending a square in the tiling of $\OO_1$ labeled by the deck transformation $g$ to the square in the tiling of $\OO_2$ labeled by $\varphi(g)$ defines a a map $\alpha: \OO_1\to \OO_2$. Observe that $\alpha$ maps neighboring squares in $\OO_1$ to neighboring squares in $\OO_2$ and thus $\alpha$ is a well-defined homeomorphism. The equation $c_1=c_2\circ\alpha$ implies the equality of the origamis.
\end{proof}

\begin{example}
	Recall that the automorphism group of the quaternion group $Q_8$ is isomorphic to the symmetric group $S_4$ and thus has order $24$. We further know that there are exactly $24$ pairs of generators $(x,y)$. Denote the set of generators by $E$. The restriction on $x,y$ is that $\{x,y\}\subseteq\{\pm i,\pm j,\pm k \}$ is a subset of size two not containing $\pm a$ for $a$ in $Q_8$. The automorphism group $\aut(Q_8)$ acts on $E$ by applying an automorphism to both components. Note that for each pair $(x,y)\in E$, the stabilizer under this action is trivial. Hence $\aut(Q_8)$ acts transitively on $E$, i.e., for any two pairs of generators $(x_1,y_1),~(x_2,y_2)$, we find an automorphism $\varphi$ of $Q_8$ such that $(\varphi(x_1),\varphi(y_1))$ is equal to $(x_2,y_2)$. By \Cref{origamis are equal}, the only $2$-origami with deck transformation group $Q_8$ is the Eierlegende Wollmilchsau (see~\Cref{eierlegende wollmilchsau}).
\end{example}

From now on, we denote a normal origami with deck transformation group $G$ defined by the pair of generators $(x,y)$ by $(G,x,y)$.

In the following, we introduce the essential definitions for this section.

\begin{definition}~
	\begin{enumerate}[\textbullet]
		\item For an origami $\mathcal{O}$ and a point $x\in\OO$ the cone angle at $x$ is $2\pi a$ for some natural number $a$. The \textbf{multiplicity} of $x$ is defined as $a$ and denoted by $\mult(x)$. A \textbf{singularity} $s\in \OO$ is a point with multiplicity larger than 1. Let $\Sigma$ denote the set of singularities.
		
		\item Let $(a_1,\ldots,a_m),(k_1,\ldots,k_m)\in \N^m$. The \textbf{stratum} $\HH({k_1}\times a_1,\ldots,{k_m}\times a_m)$ is the set of translation surfaces with $k_i$ singularities of multiplicity $a_i+1$ for $1\le i\le m$. The set of translation surfaces without a singularity, i.e., the set of tori, is denoted by $\HH(0)$. 
	\end{enumerate}
\end{definition}

Each translation surface $X$ comes with a natural holomorphic 1-form on $X$. The zeros of this 1-form are the singularities. If the order of a zero is $a$, the multiplicity of the singularity is $a+1$.

Since the squares in the tiling of an origami are glued by translations, only corners of the squares can appear as singularities. In particular, there are only finitely many singularities. We begin with a remark showing that all singularities of a normal origami have the same multiplicity. That gives a first restriction on the strata.

\begin{rem}\label[remark]{all singularities have same order}
	The deck transformation group of a normal origami acts transitively on the squares of the origami. Hence all singularities have the same multiplicity, and lie in a stratum of the form $\HH\left(0\right)$ or $\HH\left(k\times a\right)$ for $a,k\in\N$. Further, a normal origami with deck transformation group $G$ and set of singularities $\Sigma$ with $s\in \Sigma$ satisfies the following equation
	$$\sum_{s'\in \Sigma}\mult(s')=|\Sigma|\cdot \mult(s)=|G|.$$
	
	Note that either all corner points of a normal origami are singularities or all corner points are regular points.
\end{rem}

With the help of the next lemma we connect the multiplicity of singularities to a statement phrased in the language of group theory. For elements $x,y$ of a group, we denote their commutator $x\inv y\inv xy$ by $[x,y]$.

\begin{lemma}\label[lemma]{connection order singularities and commutator}
	Let $\mathcal{O}=(G,x,y)$ be a normal origami. The cover of the torus induced by the origami is unramified if and only if $[x,y]=1$. If the cover is ramified, the multiplicity of each singularity of $\mathcal{O}$ coincides with the order of $[x,y]$ in $G$.
\end{lemma}

\begin{proof}
	Let $S$ denote the square labeled by the group element $1$. Then the deck transformation $[x,y]=x\inv y\inv xy$ sends the square $S$ to the one lying $2\pi$ (with respect to the lower left corner of $S$) above $S$.
    \begin{center}
	    \captionsetup{type=figure}
	    \input{tikz/commutator-1}
        \caption{The deck transformation $[x,y]$ maps the square labeled by $1$ to the square labeled by $[x,y]$.}
    \end{center}

    Hence the deck transformation $[x,y]^m$ sends the square $S$ to the one lying $2\pi m$ above $S$ for $m\in \N$. We conclude with \Cref{all singularities have same order} that the cone angle at each corner is $2\pi\cdot~$ord$([x,y])$.
\end{proof}

\begin{rem}\label[remark]{rem: connection stratum - 2-gen gps}
	By the lemma above and \Cref{p-ori = 2-gen set}, finding a normal origami of degree $d=(a+1)k$ in the stratum $\HH(k\times a)$ is equivalent to finding a $2$-generated group of order $d$ and a generating set of size two such that the commutator of the generators has order $k$.
\end{rem}

\begin{definition}
	Let $p$ be a prime number. An origami is called \textbf{$p$-origami} if it is normal and the deck transformation group of the corresponding cover is a finite $p$-group.
\end{definition}

We will study in which strata $p$-origamis occur.
\begin{rem}\label[remark]{all singularities have same order - p group}
	By \Cref{all singularities have same order}, all singularities of a $p$-origami have the same multiplicity. Each $p$-origami outside the stratum $\HH(0)$ satisfies the equation $d=(a+1)\cdot k$, where $d$ is the degree, $a+1$ is the multiplicity of each singularity, and $k$ is the number of singularities.
	
	This reduces the possible strata significantly. All $p$-origamis of degree $p^n$ (outside the stratum $\HH(0)$) have $p^{n-k}$ singularities of multiplicity $p^k$ for $1\le k\le n$, i.e., they lie in strata of the form $\HH\left( {p^{n-k}}\times\left( p^k-1\right) \right) $.
\end{rem}

In \Cref{section: results on p-oris}, we will use the connection between group theory and the type of singularities of normal $p$-origamis to derive conclusions about the stratum they lie in.

\section{Results on \texorpdfstring{$p$}{p}-groups}\label{section: results on p-groups}

\subsection{Bounds for the exponent of commutator subgroups of \texorpdfstring{$p$}{p}-groups}\label{section: gp thy for strata results}

The geometric setting in \Cref{section: geometric motivation} motivates the study of the following problem. Given a finite $p$-group $G$ of order $p^n$, find a bound for the order of commutators $[x,y]$ with $x,y\in G$. In this section, we answer a more general question. We give a sharp bound for the exponent of the commutator subgroup. This bound is different for the prime 2 and odd primes.

We recall some basic definitions and facts from the theory of $p$-groups. Let $G$ be a finite $p$-group.
The \textbf{order} $|G|$ of $G$ is the number of its elements.
For any element $x\in G$, the \textbf{order} $\ord(x)$ of $x$ is the smallest positive integer $n$ such that $x^n=1$.
The \textbf{exponent} $\exp(G)$ of $G$ is the greatest order of any element in $G$.
The \textbf{commutator subgroup} $G'$ of $G$ is the subgroup generated by all commutators
$$
  [x,y] = x\inv y\inv x y
  \quad\text{for }x,y\in G .
$$
The \textbf{center} $Z(G)$ of $G$ is the subgroup $\{x\in G ~|~ xy=yx \ \forall y\in G\}$.
The \textbf{Frattini subgroup} $\Phi(G)$ of $G$ is the intersection of all maximal subgroups of $G$.
For each $i\in \N$, we define the \textbf{omega} and \textbf{agemo} subgroup,
\begin{align*}
  \Omega_i(G)&\vcentcolon=\langle g ~|~g\in G, g^{p^i}=1\rangle,\\
  \mho^i(G)&\vcentcolon=\langle g^{p^i} ~|~g\in G \rangle.
\end{align*}

\begin{lemma}[{\cite[Proposition 1.2.4]{LGMcK}}]\label[lemma]{frattini group props}
	Let $G$ be a finite $p$-group.
		\begin{itemize}
			\item The Frattini subgroup $\Phi(G)$ equals $G'\mho^1(G)$, the group generated by all commutators and $p$-th powers. In particular, $G/\Phi(G)$ is elementary abelian (that is, abelian and of exponent $p$).
			\item Burnside's basis theorem: A set of elements of $G$ is a (minimal) generating set if and only if the images in $G/\Phi(G)$ form a (minimal) generating set of $G/\Phi(G)$. In particular, every generating set for $G$ contains a generating set with exactly $d(G)$ elements, where $d(G)$ is the rank of the elementary abelian quotient $G/\Phi(G)$.
		\end{itemize}
\end{lemma}

This can be used to establish first bounds for the exponent of a $p$-group which holds for all prime numbers.

\begin{proposition}\label[proposition]{prop. exp(G')<n-1}
    For a finite $p$-group $G$ of order $p^n$, $\exp(G')=1$ if $n\leq 2$ or otherwise $$\exp(G')\le p^{n-2}.$$
\end{proposition}

\begin{proof}
    Any $p$-group of order $p$ is cyclic, in particular, abelian. Hence, $\exp(G')=1$, in this case.
   
	Suppose $G$ is a non-cyclic $p$-group of order $p^n$ with a minimal generating set of length $d\geq 2$. By Burnside's basis theorem, we obtain $|G/\Phi(G)|= p^d$ and thus $|\Phi(G)|= p^{n-d}$. The inclusion $G' \subseteq \Phi(G)$ implies the inequality $|G'|\le p^{n-d}$. In particular, the inequality $\exp(G')\le p^{n-d}\le p^{n-2}$ holds.
\end{proof}

\subsubsection{2-groups}
In this section, we show that the bound in \Cref{prop. exp(G')<n-1} is sharp for the prime 2. What is more, we construct $2$-generated $2$-groups with certain generators whose commutator has the desired order. These groups will be used to construct $2$-origamis in \Cref{section: strata of $2$-origamis}. We need the following lemma.

\begin{lemma}[{\cite[Hilfssatz III 1.11 a)]{hup}}] \label[lemma]{commutator-generated}
    For a group $G$ generated by a subset $S$, the commutator subgroup $G'$ is generated by the set $\{g\inv{}[x,y]g~|~x,y\in S,g\in G \}$.
\end{lemma}

\begin{proposition}\label[proposition]{constr of 2-gps}
     Let $n,k\in \NN$ with $n>2$ and $k\le n-2$. There exists a $2$-generated $2$-group $G$ of order $2^n$ with generators $x,y\in G$ such that
     $$\ord([x,y])=\exp(G')=2^k.$$
\end{proposition}

\begin{proof}
    Let $k$ be a natural number with $0\le k\le n-2$. We construct a group $G^2_{(n,k)}$ of order $2^n$ and a pair of generators $r,s$ whose commutator is of order $2^k$. The group $G^2_{(n,k)}$ is a semidirect product of two cyclic groups $C_{2^{k+1}}=\langle r \rangle$ and $C_{2^{n-k-1}}=\langle s \rangle$ of order $2^{k+1}$ and $2^{n-k-1}$, respectively. First, define the group automorphism $\alpha\colon C_{2^{k+1}}\to C_{2^{k+1}},~ r^m\mapsto r^{-m}$. Since $\alpha^{2^{n-k-1}}$ is the identity map on $C_{2^{k+1}}$, the map
	$$\varphi\colon C_{2^{n-k-1}}\to \aut(C_{2^{k+1}}),~ s^m \mapsto \alpha^m$$
	is a group homomorphism. Let $G_{(n,k)}^2$ be the semidirect product
	$$C_{2^{k+1}} \rtimes_\varphi C_{2^{n-k-1}}= \langle r,s ~|~ r^{2^{k+1}}=s^{2^{n-k-1}}=1,~s\inv rs=r\inv  \rangle.$$
	Then $G^2_{(n,k)}$ has order $2^n$.	Using the defining relations of $G^2_{(n,k)}$, we conclude
	$$[r,s]=r\inv s\inv rs=r^{-2}.$$
	Hence the commutator $[r,s]$ has order $2^{k}$.
	
	Finally we show that the commutator subgroup $(G^2_{(n,k)})'$ has exponent $2^k$. Since $G^2_{(n,k)}$ is generated by $\{r,s\}$, by \Cref{commutator-generated}, the commutator subgroup is generated by elements of the form $g\inv [r,s]g$ for $g\in G^2_{(n,k)}$. As $[r,s]=r^{-2}$ is central, each of these elements is contained in $\langle r^2 \rangle$. Hence $(G^2_{(n,k)})'$ is cyclic of order $2^k$.
\end{proof}

Note that all the $2$-groups constructed in the proof of \Cref{constr of 2-gps} are semidirect products of two cyclic groups.

\subsubsection{\texorpdfstring{$p$}{p}-groups for odd primes \texorpdfstring{$p$}{p}}

Throughout this section, let $p$ denote an odd prime. In \Cref{thm bound for exp(G')}, we introduce a much stronger bound on the exponent of the commutator subgroup which holds for odd primes. This is a generalization of a theorem by van der Waall (see~\cite[Theorem 1]{Waall}). There, the order of the commutator subgroup of finite $p$-groups is bounded under the condition that the commutator subgroup is cyclic.

Subsequently, we show in \Cref{group theory: strata p-origamis for p>2} that this bound is sharp. To this end, we construct $2$-generated $p$-groups and generators whose commutators have the desired orders. These groups are used to construct certain $p$-origamis in \Cref{section: strata of $p$-origamis}.

\begin{proposition}\label[proposition]{thm bound for exp(G')}
	For a non-trivial finite $p$-group $G$, $p$ odd, the following inequality holds
		\begin{align} \label{ineq bound for exp(G')}
		\exp(G')^2<|G|.
		\end{align}
\end{proposition}

\begin{proof}
	As the inequality holds for all cyclic $p$-groups, we may use an induction and consider a group $G$ such that the inequality holds for all groups of smaller order. 
	
	By Lemma~III 7.5 in \cite{hup}, a finite $p$-group $G$ for $p$ odd is either cyclic, or it has a normal subgroup $N\trianglelefteq G$ isomorphic to $C_p\times C_p$. In the first case, again \eqref{ineq bound for exp(G')} holds. So without loss of generality, we may assume that there exists a normal subgroup $N\trianglelefteq G$ isomorphic to $C_p\times C_p$. Define $H\vcentcolon= G/N$. By the induction hypothesis we have $\exp(H')^2<|H|$. Consider the canonical epimorphism $\varphi:G\to H$. It maps commutators of $G$ to commutators of $H$, so for $g\in G'$, the image $\varphi(g)$ lies in $H'$. Thus $\varphi(g)$ has order at most $\exp(H')$, and $g$ has order at most $\exp(N)\cdot\exp(H')=p\cdot\exp(H')$ in $G$. Hence we obtain the desired inequality
	\begin{align*}
	\exp(G')^2\le p^2\cdot\exp(H')^2 <p^2\cdot |H|=|N|\cdot|H|=|G|.
	\end{align*}
\end{proof}

\begin{corollary}\label[corollary]{cor bound for ord xy}
    Let $G$ be a $2$-generated $p$-group, $p$ odd, of order $p^n$. For generators $x,y$ of $G$, the order of their commutator obeys the inequality
    $$\ord([x,y]) < p^{\frac{n}{2}}.$$
\end{corollary}

\begin{proof}
	Let $\exp(G')=p^m, |G|=p^n$ and ord$([x,y])=p^k$. Since $k\le \exp(G')$, it is sufficient to show that $m<\frac{n}{2}$. This is equivalent to $2m<n$. By \Cref{thm bound for exp(G')}, we have $\exp(G')^2<|G|$ and thus the inequality $2m<n$ holds.
\end{proof}

As in the case of $2$-origamis, we construct for natural numbers $n,k$ with $k<\frac{n}{2}$ a $p$-group which is a semidirect product of two cyclic groups, in order to show that the proven bound is sharp. The construction given in \Cref{group theory: strata p-origamis for p>2} works similarly as the one for $2$-origamis in the proof of \Cref{strata 2-origamis}. However, the group homomorphism defining the semidirect products needs to be chosen more carefully for odd primes.

We begin with a purely number-theoretic observation which will be useful when constructing the semidirect products.

\begin{lemma}\label[lemma]{p+1 has order p^k}
	Let $p$ be an odd prime and let $k$ be a positive natural number. Then $p+1$ has order $p^k$ in $\left(\Z/p^{k+1}\Z\right) ^*$.
\end{lemma}

\begin{proof}
	First, we prove by induction on $m$ that for each $m\ge 0$, the following congruence holds
	\begin{align}\label{congruence relation}
	(1+p)^{p^m} \equiv 1+p^{m+1} \mod p^{m+2}.
	\end{align}
	This is clear for $m=0$. We assume that the congruence holds for some $m$, i.e., we have
	$$(1+p)^{p^m}=1+p^{m+1}(1+pq)$$
	for some $q\in\NN$. Using this, we compute
	\begin{align*}
	(1+p)^{p^{m+1}}&=(1+p^{m+1}(1+pq))^p\\
	&=1+ \underbrace{p\cdot p^{m+1}(1+pq)}_{\equiv~ p^{m+2}\mod p^{m+3}}+\sum_{i=2}^{p}\underbrace{\binom{p}{i}p^{(m+1)\cdot i}}_{\equiv~ 0\mod p^{m+3}}{(1+pq)^i}\\
	&\equiv 1+p^{m+2} \mod p^{m+3}.
	\end{align*}
	This shows that the congruence relation \eqref{congruence relation} is true for $m+1$. By induction, it holds for all $m\in \NN$.
	
	Choosing $m=k$ we get
	\begin{align*}
	(1+p)^{p^k}&\equiv 1+p^{k+1}\mod p^{k+2}\\
	&\equiv 1\mod p^{k+1}.
	\end{align*}
	In particular, the order of $1+p$ in $\left(\Z/p^{k+1}\Z\right) ^*$ divides $p^k$. For $1\le m<k$, we have
	\begin{align*}
	(1+p)^{p^m}&\equiv 1+p^{m+1}\mod p^{m+2}\\
	&\not\equiv 1\mod p^{m+2}.
	\end{align*}
	We conclude that
	$$(1+p)^{p^m}\not\equiv 1\mod p^{k+1}$$
	for $1\le m<k$. Thus the order of $1+p$ is $p^k$.
\end{proof}

\begin{proposition}\label[proposition]{group theory: strata p-origamis for p>2}
	Let $n,k\in \NN$ with $k<\frac{n}{2}$. There exists a $2$-generated $p$-group $G$ of order $p^n$ with generators $x,y\in G$ such that
     $$\ord([x,y])=\exp(G')=p^k.$$
\end{proposition}

\begin{proof}
    Fix a positive natural number $n$ and let $k$ be an integer with $0\le k<\frac{n}{2}$. The group $G^p_{(n,k)}$ is constructed as a semidirect product of two cyclic groups $C_{p^{k+1}}=\langle r \rangle$ and $C_{p^{n-k-1}}=\langle s \rangle$ of order $p^{k+1}$ and $p^{n-k-1}$, respectively. First, consider the automorphism group of $C_{p^{k+1}}$. From elementary group theory, we know that
	$$\aut(C_{p^{k+1}}) \cong \left( \Z/p^{k+1}\Z \right)^*\cong C_{\varphi(p^{k+1})}= C_{p^k(p-1)}.$$
	The map
	$$\alpha\colon C_{p^{k+1}}\to C_{p^{k+1}}, r^m\mapsto r^{m\cdot(p+1)}$$
	defines a group automorphism, since $p$ and $p+1$ are coprime.
	
	Now we consider the map
	$$\varphi\colon C_{p^{n-k-1}}\to \aut(C_{p^{k+1}}), s^m\mapsto \alpha^m.$$
	We claim that $\varphi$ is a well-defined group homomorphism. To see this, we need show that $\alpha^{p^{n-k-1}}$ is the identity map on $C_{p^{k+1}}$. This follows, since $k\leq \frac{n-1}2$ implies $n-k-1\geq k$, so $(p+1)^{p^{n-k-1}}\equiv 1$ mod $p^{k+1}$, again using \Cref{p+1 has order p^k}. The congruence follows because $p+1$ has order $p^k$ in $\left( \Z/p^{k+1}\Z \right)^*$ by \Cref{p+1 has order p^k}. Write
	$$\left( p+1\right) ^{p^{n-k-1}}= j\cdot p^{k+1}+1$$
	for some natural number $j$. Then we have
	\begin{align*}
	\alpha^{p^{n-k-1}}(r^i)&=r^{i\cdot(p+1)^{p^{n-k-1}}}\\
	&=r^{i\cdot j\cdot p^{k+1}}\cdot r^i\\
	&=r^i
	\end{align*}
	for each $1\le i\le p^{k+1}$. Since $r$ has order $p^{k+1}$ in $C_{p^{k+1}}$, the last equality follows. We conclude that $\alpha^{p^{n-k-1}}=\id$.
	
	Let $G^p_{(n,k)}$ be the semidirect product
	$$C_{p^{k+1}} \rtimes_\varphi C_{p^{n-k-1}}= \langle r,s ~|~ r^{p^{k+1}}=s^{p^{n-k-1}}=1,~s\inv rs=r^{p+1} \rangle.$$
	Then $G^p_{(n,k)}$ has order $p^n$. We claim that $G^p_{(n,k)}$ together with the pair of generators $(r,s)$ has the desired properties. Using the defining relations of $G^p_{(n,k)}$, we conclude
	$$[r,s]=r\inv s\inv rs=r^{p}.$$
	In particular, the commutator $[r,s]$ has order $p^{k}$.
	
	Finally we show that the commutator subgroup of $G^p_{(n,k)}$ has exponent $p^k$. Since $G^p_{(n,k)}$ is generated by $\{r,s\}$, the commutator subgroup is generated by elements of the form $g\inv [r,s]g$ for $g\in G^p_{(n,k)}$ by \Cref{commutator-generated}. Each of these elements is contained in $\langle r^p \rangle$. Hence $(G^p_{(n,k)})'$ is cyclic of order $p^k$.
\end{proof}

\subsection{On the order of certain commutators in \texorpdfstring{$2$}{2}-generated \texorpdfstring{$p$}{p}-groups}\label{section: gp thy iso deck groups}

In this section, we study a second question that arises from the geometric setting in \Cref{section: geometric motivation}. Recall that the group of deck transformations of a normal origami is always a finite $2$-generated group and that two normal origamis with isomorphic group of deck transformations $(G,x_1,y_1)$ and $(G,x_2,y_2)$ lie in the same stratum if and only if the orders of the commutators $[x_1,y_1]$ and $[x_2,y_2]$ agree.

We first note the the possible strata for normal origamis with a fixed deck group depend only on its isoclinism class. We recall that isoclinism is an equivalence relation for groups generalizing isomorphism. For its definition we use the observation that the commutator in any group $G$ induces a well-defined map $G/Z(G)\times G/Z(G)\to G$.

\begin{definition}[\cite{Hall}] Two groups $G_1,G_2$ are \textbf{isoclinic} if there are isomorphisms $\phi\colon G_1/Z(G_1)\to G_2/Z(G_2)$ and $\psi\colon G'_1\to G'_2$ which are compatible with the commutator maps in the sense that the following is a commutative diagram:
\begin{align*}
\xymatrix{
	G_1/Z(G_1)\times G_1/Z(G_1)
	\ar[d]_{\phi\times\phi}
	\ar[r]^{~~~~~~~~~~~~~[\cdot,\cdot]}
	&
	G'_1
	\ar[d]_{\psi}
    \\
	G_2/Z(G_2)\times G_2/Z(G_2)
	\ar[r]^{~~~~~~~~~~~~[\cdot,\cdot]}
	&
	G'_2
}.
\end{align*}
\end{definition}

In particular, all abelian groups are isoclinic.

\begin{lemma}\label[lemma]{lem-isoclinic-groups} The set of possible commutator orders $\ord([x,y])$ for generators $x,y$ of a $2$-generated group $G$ depends only on the isoclinism class of $G$.
\end{lemma}

\begin{proof} Assume $G_1,G_2$ are isoclinic groups and $[x,y]$ has order $n$ for generators $x,y$ of $G_1$. Then $n$ is also the order of $\psi([x, y])=[\phi(\overline x),\phi(\overline y)]$, where $\overline x,\overline y$ are the images in $G_1/Z(G_1)$. Thus $\ord([x',y'])=\ord([x,y])$ for any $x',y'$ in $G_2$ such that $x'\equiv\phi(x)$ and $y'\equiv\phi(y)$ modulo $Z(G_2)$.

If $Z(G_2)$ does not lie in $\Phi(G_2)$, then a central element can be chosen as one of two generators of $G_2$ (see \Cref{frattini group props}), so $G_2$ is abelian and $G_1$ is abelian, and the only possible commutator order is $1$ in each of these groups.

Otherwise, $Z(G_2)$ does lie in $\Phi(G_2)$ and the images of $x',y'$ in $G_2/\Phi(G_2)$ generate the quotient, so $x',y'$ are generators of $G_2$ and the order of $[x',y']$ equals the order of $[x,y]$. Since isoclinism is a reflexive relation, this proves the assertion.
\end{proof}

\begin{corollary}\label[corollary]{cor-dihedral-groups} For each $n\geq 1$, the dihedral group, the generalized quaternion group, and the semidihedral group with $2^n$ elements have the same set of possible commutator orders $\ord([x,y])$ for generators $x,y$.
\end{corollary}

\begin{proof} They are well-known to be isoclinic, see for instance \cite[§29 Exercise 4]{Berkovich}.
\end{proof}

We will return to these groups, recall their definition, and compute the set of possible commutator orders in \Cref{maximal p-groups}.

We noted that the possible strata of $p$-origami with a given deck group depends on the possible commutator orders for pairs of generators of this group. We will see that, in fact, for many groups there is only one stratum possible. To study such groups, we first translate this property into the language of group theory.

\begin{definition}\label[definition]{def property C}
We say that a finite $2$-generated group $G$ has \textbf{\property}, if there exists a natural number $n$ such that for each $2$-generating set $\{x,y\}$ of $G$ the order of $[x,y]$ equals $n$.
\end{definition}

\begin{question}\label[question]{qu gp th iso gp same stratum}
	Which finite $2$-generated $p$-groups have \property?
\end{question}

For a large class of $p$-groups we prove \property. However, in \Cref{section: counterexample} we construct for each prime $p$ a finite $p$-group with generating sets $\{x,y\}$ and $\{x',y'\}$ such that $\ord([x,y])\neq\ord([x',y'])$.

In a first example, we show that most alternating groups -- which are not $p$-groups -- do not have \property. We use this example to construct origamis in \Cref{section: results on p-oris}.

\begin{example}\label[example]{ex. iso deck groups not same stratum}
	For $n\in \mathbb{N}_{\ge5}$, we consider the alternating group $A_n$ with the following pairs of generators: $((1,2,\dots,n-1,n), (1,2,3))$ and $((3,4,\dots, n-1,n), (1,3)(2,4))$. The orders of the commutators
	\begin{align*}
	    [(1,2,\dots,n-1,n),(1,2,3)]&=(1,2,4)\\
	    [((3,4,\dots, n-1,n),(1,3)(2,4))]&=(1,2,5,4,3)
	\end{align*}
	are $3$ and $5$, respectively. Hence there are two pairs of generators such that the order of their commutator is different and $A_n$ does not have \property for $n\geq 5$. Notice that those $A_n$ are not $p$-groups, since their order is $\tfrac{n!}{2}$.
	
	Further, note that we multiply permutations from the left because we label the squares of a normal origami by multiplying generators of the deck group from the right. 
\end{example}

In the following, we will prove \property for certain families of $p$-groups.

\subsubsection{Regular and order-closed \texorpdfstring{$p$}{p}-groups} \label{sec-order-closed}

We begin by stating some basic properties of regular $p$-groups. Recall that a $p$-group $G$ is \textbf{regular} if for each $g,h\in G$ and $i\in \N$, there exists some $c\in \left\langle g,h\right\rangle$ such that
\begin{align*}
(gh)^{p^i}=g^{p^i}h^{p^i}c.
\end{align*} 
Note that the commutator subgroup of a regular $p$-group is regular.

We call a $p$-group $G$ \textbf{weakly order-closed} if the product of elements of order at most $p^k$ has order at most $p^k$ for any $k\geq 0$. In the literature, $p$-groups for which all sections (i.e., subquotients) are weakly order-closed according to our definition have been studied and are called \textbf{order-closed} $p$-groups (see \cite{Mann}). Clearly, order-closed $p$-groups are weakly order-closed, and one can verify that all subgroups of a weakly order-closed group are so, as well. Hence a $p$-group is order-closed if and only if all its subgroups are weakly order-closed. The class of $p$-groups we call weakly order-closed has been called $\mathcal{O}_p$ in \cite{wil}.

\begin{lemma}[see~\cite{LGMcK}, Lemma 1.2.13]\label[lemma]{props of regular gps}~
	\begin{itemize}
		\item[1)] Any regular $p$-group is order-closed, and hence also weakly order-closed.
		\item[2)] A $2$-group is regular if and only if it is abelian.
	\end{itemize}
\end{lemma}

\begin{corollary}
	Let $G$ be a weakly order-closed $p$-group. If $x_1,\dots,x_r\in G$ generate $G$, then the exponent of $G$ is equal to the maximum of the orders $ \mathrm{ord}(x_i), ~1\le i\le r$.
\end{corollary}

\begin{proof}
	Note that the orders of $x_i$ and $x_i\inv $ are equal. As every group element can be written as a word in $\{x_i,~x_i\inv ~|~1\le i \le r\}$, the claim follows.
\end{proof}

\begin{proposition}\label[proposition]{iso regular gp same stratum}\label[proposition]{G' weakly order-closed implies property}
	Any finite $2$-generated $p$-group with weakly order-closed commutator subgroup has \property. In particular, any finite $2$-generated $p$-group with regular commutator subgroup has \property.
\end{proposition}

\begin{proof}
	Let $G$ be a finite $2$-generated $p$-group with regular commutator subgroup. Further, let $x,y$ and $x',y'$ be two pairs of generators of $G$. Hence, by \Cref{commutator-generated}, the commutator subgroup $G'$ is generated by each of the sets
	\begin{align*}
	&\left\lbrace g\inv [x,y]g~|~g\in G \right\rbrace,\\
	&\left\lbrace g\inv [x',y']g~|~g\in G \right\rbrace.
	\end{align*}
	We have
	\begin{align*}
	\mathrm{ord}(g\inv [x,y]g)&=\mathrm{ord}([x,y]),\\ \mathrm{ord}(g\inv [x',y']g)&=\mathrm{ord}([x',y']).
	\end{align*}
	By the previous lemma, $G'$ being weakly order-closed implies that
	$$\mathrm{ord}([x,y])=\mathrm{exp}(G')=\mathrm{ord}([x',y']).$$
	Hence the orders of the commutators $[x,y]$ and $[x',y']$ coincide.
\end{proof}

As we will discuss in \Cref{section infinite origamis}, a version of this result holds for pro-$p$ groups.

\subsubsection{\texorpdfstring{$p$}{p}-groups of maximal class}\label{maximal p-groups}
In this section, we prove that \property holds for $p$-groups of maximal class.

We recall that the \textbf{lower central series} of a group $G$ is the series of subgroups
\begin{align*}
\gamma_1(G)\ge \gamma_2(G)\ge \gamma_3(G)\ge \dots,
\end{align*}
where $\gamma_1(G)\vcentcolon=G$ and $\gamma_i(G)\vcentcolon=[\gamma_{i-1}(G),G]$ for $i>1$. The \textbf{nilpotency class} of $G$ is $c$ if $$\gamma_c(G) \gneq \gamma_{c+1}(G)=\left\langle 1\right\rangle .$$

A $p$-group of order $p^n$ has nilpotency class at most $n-1$, and is called of \textbf{maximal class} in that case.

We treat 2-groups separately from the $p$-groups for odd primes $p$. The $2$-groups of maximal class are completely classified by the following theorem.

\begin{proposition}[{\cite[Kapitel~III, Satz~11.9]{hup}}]\label[proposition]{2-gps of max class}
	Each $2$-group of maximal class and order $2^n$ is isomorphic to one of the following groups
		\begin{itemize}
			\item a dihedral group, i.e., a group given by the presentation
			\begin{align*}
			D_{2^n}=\left\langle r,s~|~r^{2^{n-1}}=s^2=1,~s\inv rs=r\inv \right\rangle,
			\end{align*}
			\item a generalized quaternion group, i.e., a group given by the presentation
			\begin{align*}
			Q_{2^n}=\left\langle r,s~|~r^{2^{n-1}}=1,~s^2=r^{2^{n-2}},~s\inv rs=r\inv \right\rangle,
			\end{align*}
			\item a semidihedral group, i.e., a group given by the presentation
			\begin{align*}
			SD_{2^n}=\left\langle r,s~|~r^{2^{n-1}}=s^2=1,~s\inv rs=r^{2^{n-2}-1}\right\rangle.
			\end{align*}
		\end{itemize}
\end{proposition}

We can now strengthen the result in \Cref{cor-dihedral-groups}. There we already noted that in dihedral groups, generalized quaternion groups, and semidihedral groups of order $2^n$, i.e., all $2$-groups of maximal class, the commutator of any pair of generators has a unique order, since these groups are isoclinic.

\begin{lemma}\label[lemma]{iso 2-gp of max class same stratum}
	Any $2$-generated $2$-group of maximal class and of order $2^n$ has \property and for each pair of generators $x,y,$ the order of the commutator $[x,y]$ is $2^{n-2}$.
\end{lemma}

\begin{proof}
	Let $G$ be a $2$-group of maximal class and of order $2^n$. By \Cref{2-gps of max class}, it is isomorphic to a dihedral group, quaternion group or semidihedral group. In each case, we show that the commutator subgroup is regular and that for a pair of generators $r,s$ the commutator $[r,s]$ has order $2^{n-2}$. By \Cref{iso regular gp same stratum}, this proves the claim.
	
	For the semidihedral group $SD_{2^n}$, the commutator subgroup is generated by the set $\{ g\inv [r,s]g~|~g\in SD_{2^n}\}$. Since the commutator $[r,s]$ equals $r^{2^{n-2}-2}$, the commutator subgroup is cyclic. In particular, it is abelian and thus regular. The order of the commutator $[r,s]$ is $2^{n-2}$.
	
	For the dihedral group $D_{2^n}$ and the generalized quaternion group $Q_{2^n}$, the commutator subgroup is generated by $[r,s]=r^{-2}$. Hence the commutator subgroup is abelian and thus regular. The commutator $[r,s]$ has order $2^{n-2}$.
\end{proof}

We turn our attention to the odd primes.

\begin{lemma}\label[lemma]{lem maximal class p^n 5<=n<=p+1}
	Let $G$ be a $2$-generated $p$-group of maximal class and order $p^n$ for $5\le n\le {p+1}$ and an odd prime $p$. Then $G$ has \property and for each pair of generators $x,y,$ the order of the commutator $[x,y]$ equals $p$.
\end{lemma}

\begin{proof}
	Assume $G$ is a group of order $p^n$ of maximal class for $5\le n\le p+1$. Then the commutator subgroup $G'$ has exponent $p$ by \cite[Kapitel~III, Hilfssatz~14.14]{hup}. Hence all elements of $G'$ have either order $p$ or order 1. Let $(x,y)$ be a pair of generators of $G$. Then $G'$ is generated by elements of the form $g\inv [x,y]g$ for $g\in G$. Since $G$ is of maximal class, it is non abelian. We conclude that $[x,y]\ne 1$ has order $p$.
\end{proof}

\begin{lemma}\label[lemma]{iso p-gp of max class same stratum p odd}
	Any finite $2$-generated $p$-group of maximal class for $p$ odd has \property.
\end{lemma}

\begin{proof}
	Let $G$ be any group of order $p^n$. If $n\le4$ then the order of the commutator subgroup $G'$ is smaller than $p^4$. Hence we have $|G'|\le p^p$. By \cite[Kapitel III Satz 10.2 b)]{hup}, the commutator subgroup is regular. Using \Cref{props of regular gps}, the claim follows from \Cref{G' weakly order-closed implies property}. In \Cref{lem maximal class p^n 5<=n<=p+1}, we showed the claim for groups of order $p^n$ of maximal class with $5\le n\le p+1$.
	
	Now assume $G$ has order $p^n$ and is of maximal class for $n> p+1$. Then there exists a maximal subgroup $H$ of $G$ which is regular (see~\cite[Kapitel~III, Satz~14.22]{hup}). Recall that the Frattini subgroup $\Phi(G)$ is the intersection of the set of maximal subgroups. We have the following inclusions
	$$ G'\subseteq \Phi(G)\subseteq H.$$
	Since subgroups of regular groups are regular as well, the commutator subgroup $G'$ is regular. By \Cref{iso regular gp same stratum}, the claim follows.
\end{proof}

We obtain the following proposition from  \Cref{iso 2-gp of max class same stratum} and \Cref{iso p-gp of max class same stratum p odd}.

\begin{proposition}\label[proposition]{iso p-gp of max class same stratum}
	Any finite $2$-generated $p$-group of maximal class has \property.
\end{proposition}

\subsubsection{Powerful \texorpdfstring{$p$}{p}-groups}

We introduce some further basics from the theory of $p$-groups to show that \property holds for the class of powerful $p$-groups.

A $p$-group $G$ is called \textbf{powerful} if either $p$ is odd and $G'\subseteq \mho^1(G)$, or $p=2$ and $G'\subseteq \mho^2(G)=\left\langle g^4~|~g\in G \right\rangle$. A normal subgroup $N\trianglelefteq G$ is \textbf{powerfully embedded} in $G$ if either $p$ is odd and $[N,G]\subseteq\mho^1(N)$, or $p=2$ and $[N,G]\subseteq\mho^2(N)$. In particular, any powerfully embedded $p$-group is powerful.

\begin{corollary}
	Any finite $2$-generated powerful $p$-group $G$ has a cyclic commutator subgroup $G'$.
\end{corollary}

\begin{proof}
    Let $G$ be a powerful $p$-group with $2$-generating set $\{x,y\}$. Then the commutator subgroup $G'$ is powerfully embedded in $G$ by \cite[Theorem 1.1. and Theorem 4.1.1.]{LM}. Further, $G'$ is generated by all elements of the form $g\inv [x,y] g$ for $g\in G$ by \Cref{commutator-generated}. Now \cite[Theorem 1.10. and Theorem 4.1.10.]{LM} state that if a powerfully embedded subgroup of a powerful $p$-group is the normal closure of a subset, then it is generated by this subset. Thus, we conclude that $G'$ is generated by $[x,y]$, i.e., $G'$ is cyclic.
\end{proof}

\begin{corollary}\label[corollary]{iso powerful p-gp same stratum}
	Any finite $2$-generated powerful $p$-group has \property.
\end{corollary}

\begin{proof}
    $G'$ is cyclic, so in particular, $G'$ is regular. Using \Cref{iso regular gp same stratum}, the claim follows.
\end{proof}

\begin{proposition}\label[proposition]{iso p-gp powerful G' same stratum}
	Let $G$ be a finite $2$-generated $p$-group such that $G'$ is powerful. Then $G$ has \property.
\end{proposition}

\begin{proof}
    Let $G$ be a finite $2$-generated $p$-group such that $G'$ is powerful. Let $\{x,y\}$ be a generating set. Denote the order of $[x,y]$ by $p^m$. Recall that $G'$ is generated by the set $\{g\inv [x,y]g~|~g\in G\}$ by \Cref{commutator-generated}. Since $G'$ is powerful, we apply \cite[Theorem 1.9. and Theorem 4.1.9.]{LM} stating that in a powerful $p$-group, any agemo subgroup is generated by the corresponding powers of a given set of generators, and deduce
    \begin{align*}
        \mho^m(G')&
        =\langle  g^{p^m}~|~g\in G'\rangle
        =\langle \left(g\inv [x,y] g\right)^{p^m}|~g\in G\rangle
        =\{1\}.
    \end{align*}
    In particular, the exponent of $G'$ equals ${p^m}$, and hence the order of $[x,y]$. As $\{x,y\}$ was an arbitrary pair of generators of $G$, this proves the claim.
\end{proof}

As the commutator subgroup of a powerful $p$-group is powerful by \cite[Theorem 1.1]{LM}, \Cref{iso powerful p-gp same stratum} is also a consequence of \Cref{iso p-gp powerful G' same stratum}.

We will see in \Cref{section infinite origamis}, that the result of the proposition can be extended to pro-$p$ groups.

\subsubsection{Power-closed \texorpdfstring{$p$}{p}-groups} Recall from \Cref{sec-order-closed} that we call a $p$-group weakly order-closed, if products of elements of order at most $p^k$ have order at most $p^k$ for all $k\geq 0$. We call a $p$-group \textbf{weakly power-closed}, if products of $p^k$-th powers are $p^k$-th powers for all $k\geq 0$. Such groups generalize \textbf{power-closed} $p$-groups (\cite{Mann}), for which all sections (i.e., subquotients) have to be weakly power-closed in our sense. As quotients of weakly power-closed groups are automatically weakly power-closed, a $p$-group is power-closed if and only if all its subgroups are weakly power-closed. In \cite{wil}, the class of $p$-groups we call weakly power-closed has been called $\mathcal{P}_p$. 

It is known that order-closed $p$-groups generalize regular $p$-groups, and power-closed $p$-groups generalize order-closed $p$-groups (see \cite{Mann}).

Recall that in \Cref{G' weakly order-closed implies property}, we have shown that $2$-generated $p$-groups $G$ for which $G'$ is weakly order-closed have \property.

\begin{example}\label[example]{ex. power-closed gp not property C} Using the \texttt{GAP} code in \Cref{code-powerclosed} (in \Cref{sec-code}), we find instances of a $2$-group $G$ such that $G'$ is weakly power-closed, but which does have generators $x,y$ such that $\ord([x,y])\neq\ord([x,y^3])$. As $G$ is a $2$-group, $x,y^3$ is a pair of generators, as well, so $G$ does not have \property.

An example for such a situation is the case where $x,y\in S_{16}$ are the permutations 
\begin{align*}
    (1,13,2,14)(3,16,4,15)(5,9,7,11,6,10,8,12)
    \ ,\\
    (1,16,6,11,4,14,7,9,2,15,5,12,3,13,8,10)
    \ ,
\end{align*}
and $G=\langle x,y\rangle$; then
\begin{align*}
[x,y] &= (1,5,2,6)(3,7,4,8)(9,13,10,14)(11,16,12,15) 
\ ,\\
[x,y^3] &= (1,6)(2,5)(3,7)(4,8)(9,15)(10,16)(11,14)(12,13)
\end{align*}
have order $4$ and $2$, respectively.

However, this counterexample is not a power-closed $p$-group, since it has subgroups which are not weakly power-closed. Such subgroups can be found using computer algebra codes, as the ones described in \Cref{sec-code}.
\end{example}

\begin{corollary}\label[corollary]{cor. weakly power-closed gp not property C}
    There are $2$-groups with weakly power-closed commutator subgroup which do not have \property.
\end{corollary}

However, let us contrast this with the case of minimal non-power-closed or minimal non-order-closed $p$-groups, where minimal means that all proper sections are power-closed or order-closed, respectively.

\begin{lemma}\label[lemma]{minimal non-P_i gps}
	Let $G$ be a minimal non-power-closed $p$-group or a minimal non-order-closed $p$-group. Then $G$ has \property and for each pair of generators $x,y,$ the order of the commutator $[x,y]$ equals $p$.
\end{lemma}

\begin{proof}
    Let $G$ be a minimal non-power-closed $p$-groups or minimal non-order-closed $p$-group. By \cite[Theorem 3 and Theorem 6]{Mann}, the Frattini group $\Phi(G)$ has exponent $p$ and the center $Z(G)\neq G$. As $G'$ is contained in $\Phi(G)$, we conclude that $G'$ has exponent $p$. Hence $G'$ is regular. \Cref{iso regular gp same stratum} implies that $G$ has \property. For a pair of generators $x,y$ of $G$, the commutator subgroup $G'$ is generated by all conjugates of $[x,y]$. Since $G'$ is not trivial, the commutator $[x,y]$ has exponent $p$.
\end{proof}

Let us summarize our results on groups with \property.

\begin{theorem}\label[theorem]{thm-diagram} The implications and the nonimplication visualized in the following diagram are true for any finite $2$-generated $p$-group $G$. In particular, all $p$-groups falling into the classes above the gray line have \property.
\diagram
\end{theorem}

\begin{proof} The implications regular $\Rightarrow$ order-closed $\Rightarrow$ power-closed are due to \cite{Mann}. Each of these properties is inherited by subgroups, so by $G'$ from $G$. By our definition, the properties weakly order-closed and power-closed generalize order-closed and power-closed, respectively. $p$-groups of maximal class or powerful $p$-groups have regular commutator subgroups as explained in the proof of \Cref{iso p-gp of max class same stratum p odd} and \Cref{iso powerful p-gp same stratum}. The commutator subgroup of a powerful $p$-group is powerful by \cite[Theorem 1.1, Theorem 4.1.1]{LM}.

The results that $G'$ powerful or weakly order-closed imply \property, are \Cref{iso p-gp powerful G' same stratum} and \Cref{G' weakly order-closed implies property}. The fact that $G'$ weakly power-closed does not imply \property is \Cref{cor. weakly power-closed gp not property C}. 
\end{proof}

\begin{rem}\label[remark]{rem-small-groups} We note that, in particular, all $2$-generated $p$-groups of order at most $p^{p+2}$ or of nilpotency class at most $p$ have \property. In the first case, $G'$ has order at most $p^p$, because it lies in the Frattini subgroup which has index $p^2$ in $G$, so $G'$ is regular by \cite[Satz 10.2 b)]{hup}. Similarly, in the second case, $G$ is regular by \cite[Satz 10.2 a)]{hup}.
\end{rem}

It is an open questions, whether $G$ or $G'$ being power-closed implies \property for a $p$-group $G$. To study this and similar questions, we offer a reduction argument.

\begin{lemma}\label[lemma]{2 cyclic subgps of order p}
    Let $\mathcal{F}$ be a family of finite $2$-generated $p$-groups which is closed under quotients and such that \property holds for all members of $\mathcal{F}$ with cyclic center. Then \property holds for all members of $\mathcal{F}$.
\end{lemma}
	
\begin{proof}
    Assume \property holds for all abelian groups and all members of $\mathcal{F}$ with cyclic center, and consider a non-abelian group $G\in\mathcal{F}$ whose center is not cyclic such that \property holds for any member of $\mathcal{F}$ of smaller cardinality.
    
	Since the center of $G$ has rank at least two, we can choose two trivially intersecting central cyclic subgroups $N_1,N_2$ of order $p$. We consider pairs of generators $x,y$ and $x',y'$ of $G$, and we define
	\begin{align*}
	&c\vcentcolon=[x,y],~~\phantom{..} m\vcentcolon= \ord([x,y]), \\
	&c'\vcentcolon=[x',y'],~ m'\vcentcolon= \ord([x',y']).
	\end{align*} 
	Without loss of generality, we may assume $m\geq m'\geq p$.
	Now $c^{\frac{m}{p}}$ has order $p$. Since $N_1$ and $N_2$ intersect trivially, $c^{\frac{m}{p}}$ lies in at most one of these subgroups. We may assume that $c^{\frac{m}{p}}\notin N_1$. Consider the canonical epimorphism $\overline{\phantom{X}}:G\to G/N_1$. The order of $\overline{c}$ is equal to $m$ because $c^{\frac{m}{p}}\notin N_1$. Now $\overline{x}, \overline{y}$ and $\overline{x'}, \overline{y'}$ are pairs of generators of the $p$-group $G/N_1$. By our assumptions, $G/N_1$ has \property, so
	$$m=\ord(\overline{c})=\ord(\overline{c'}).$$
	Thus we conclude that the $m'$ is either $m$ or $m\cdot p$, so as $m\geq m'$, we conclude that $m$ and $m'$ coincide, and $G$ has \property.
	
	An induction using this argument proves the assertion.
\end{proof}

\begin{corollary} $G$ being power-closed or $G'$ being power-closed implies \property if it implies \property together with the assumption that $G$ has cyclic center.
\end{corollary}

\begin{proof} The family of power-closed $p$ groups is closed under taking arbitrary quotients. The same is true for the family of $p$-groups with power-closed commutator subgroup, since for any group, the commutator subgroup of a quotient is a quotient of the commutator subgroup. 
\end{proof}

\subsubsection{A family of counterexamples}\label{section: counterexample} We conclude our discussion of \property by constructing $2$-generated $p$-groups for all primes $p$ which do not have \property: for each such group, we exhibit two pairs of generators with different commutator order.

As these will be realized as certain subgroups of a $p$-Sylow subgroup of the symmetric group on $p^r$ letters, for some $r\geq 0$, let us first recall a description of the relevant Sylow subgroups. Recall that we multiply permutations from the left.

\begin{definition} For any prime $p$ and any $r\geq 0$, let $\mathbf{P_{p,r}}$ be the subgroup of the symmetric group $S_{p^r}$ generated by the elements
$$\{e_{i,j} \mid 1\leq i\leq r, 1\leq j\leq p^{i-1}\}\ ,$$
where each $e_{i,j}$ is a product of disjoint $p$-cycles defined by
\begin{equation}\label{eq-def-eij}
e_{i,j} :=
\prod_{k=p^{r+1-i} (j-1)+1}
^{p^{r+1-i} (j-1)+p^{r-i}}
(k, k+p^{r-i}, \dots,  k+p^{r-i}(p-1) )
\ .
\end{equation}
\end{definition}

We can think of the permutations $e_{i,j}$ as graph automorphisms of the perfect $p$-ary tree with $p^r$ leaves (as in \Cref{fig:tree} below for $p=r=3$). Let $N$ be the $j$-th node of this tree at level $i-1$ (where the root node is at level $0$), then we define a tree automorphism by fixing all nodes which are not descendants of $N$, while rotating the subtrees starting at the $p$ direct descendants of $N$ cyclically to the right. The action of this automorphism on the $p^r$ leaves of the tree corresponds exactly to the action of $e_{i,j}$.
\begin{figure}[ht]
    \centering
\begin{tikzpicture}[remember picture,draw=black]
\def\r{4}
\def\p{3}
\newcommand\pos[2]{(%
{  ( #2 - 1 - (pow(\p,#1-1)-1)/2 ) * pow(\p,\r-(#1)) * 0.4  },%
{  -(#1) * 1  }%
)}
\foreach \i in {1,...,\r} {
  \pgfmathtruncatemacro\n{pow(\p,\i-1)}
  \foreach \j in {1,...,\n} {
    \ifnum\i>1
      \begin{scope}[on background layer]
      \draw \pos{\i}{\j} -- \pos{\i-1}{1+div(\j-1,\p)};
      \end{scope}
    \fi
    \draw[fill=white] \pos{\i}{\j} circle (0.6mm);
  }
}
\end{tikzpicture}
    \caption{A perfect tertiary tree with $3^3$ leaves. Its automorphism group contains a group isomorphic to $P_{3,3}$.}
    \label[figure]{fig:tree}
\end{figure}
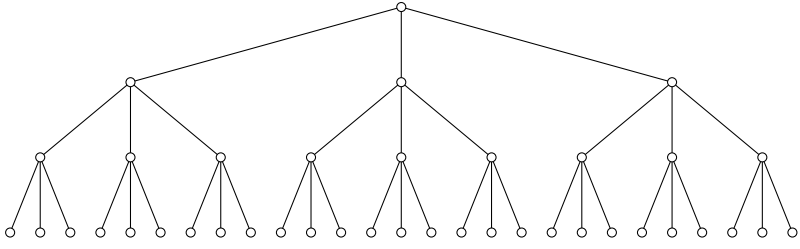

From \Cref{eq-def-eij} we get the conjugation relation 
\begin{equation}\label{eq-conjugation-eij}
    e_{i,j}\inv e_{k,l} e_{i,j} = e_{k,l'}
\end{equation}
for $i\leq k$, where $l'$ is defined by
$$ (j-1)p^{k-i} < l' \leq jp^{k-i}
\quad\text{and}\quad
l'-l \equiv p^{k-i-1} \mod p^{k-i}
$$
if $i<k$ and $(j-1)p^{k-i} < l \leq jp^{k-i}$, and $l=l'$ otherwise.

This implies that $P_{p,r}$ is, in fact, generated by $(e_{i,1})_{1\leq i\leq r}$. We can check that $P_{p,r}$ is isomorphic to the $r$-fold wreath product of $C_p$ by identifying the respective generators. So $P_{p,r}$ is a Sylow $p$-subgroup of $S_{p^r}$.

\begin{lemma}\label[lemma]{gp without property (C)}
For any prime $p$, there exist elements $x,x',y\in P_{p,4}$ such that
$$ H_p\vcentcolon=\langle x,y\rangle = \langle x',y\rangle \quad\text{and}\quad \ord([y,x']) = p \neq p^2 = \ord([y,x])
\ .
$$
In particular, the $p$-group $H_p$ does not have \property.
\end{lemma}

\begin{proof}

We define
$$ x := e_{1,1} e_{2,1} \ ,\quad y:= e_{1,1} e_{2,1} e_{3,1} e_{4,1+p^2} \quad \in P_{p,4}
\ .
$$

Assume $p>3$, then by \Cref{eq-conjugation-eij},
\begin{align*}
   x\inv y x
   &= e_{2,1}\inv (e_{1,1} e_{2,2} e_{3,1+p} e_{4,1+2p^2}) e_{2,1}
   = e_{1,1} e_{2,1} e_{3,1+p} e_{4,1+2p^2}
   \ ,
   \\
   x^{-2} y x^2
   &= e_{2,1}\inv (e_{1,1} e_{2,2} e_{3,1+2p} e_{4,1+3p^2}) e_{2,1} 
   = e_{1,1} e_{2,1} e_{3,1+2p} e_{4,1+3p^2}
   \ ,
   \\
   \text{so}\quad
   [y,x]
   &= y\inv x\inv y x = e_{4,1+p^2}\inv e_{3,1}\inv e_{3,1+p} e_{4,1+2p^2}
   \ ,
   \\
   [y,x^2]
   &= y\inv x^{-2} y x^2 = e_{4,1+p^2}\inv e_{3,1}\inv e_{3,1+2p} e_{4,1+3p^2}
   \ .
\end{align*}

We compute the restriction of the permutation $[y,x]$ to the set $S:=\{p^3+1,\dots, p^3+p^2\}$:
\begin{align*}
   [y,x] |_S
&= (e_{4,1+p^2}\inv e_{3,1}\inv e_{3,1+p} e_{4,1+2p^2}) |_S
= (e_{4,1+p^2}\inv e_{3,1+p} ) |_S
\\
&= (p^3+1, p^3+2, \dots, p^3+p)\inv
(p^3+1, p^3+p+1, \dots, p^3+p(p-1)+1) \\
&\quad\quad\dots (p^3+p, p^3+2p, \dots, p^3+p^2)
\\
&= (
p^3+p, p^3+2p-1, \dots, p^3+p^2-1,  \\
&\quad p^3+p-1, p^3+2p-2, \dots, p^3+p^2-2, \\
&\quad \dots, \\
&\quad p^3+1, p^3+2p, \dots, p^3+p^2)
\ ,
\end{align*}
a $p^2$-cycle. Hence $\ord([y,x])\geq p^2$. As all $p$-cycles occurring in $e_{3,1}$ or $e_{4,1+2p^2}$ are disjoint from each other and from the set $S$, indeed, $\ord([y,x])=p^2$. At the same time, our formula for $[y,x^2]$ exhibits it as a product of disjoint $p$-cycles, hence $\ord([y,x^2])=p$. So we have shown
$$
\ord([y,x^2])=p \neq p^2 = \ord([y,x])
\ ,
$$
and since $p\neq 2$, $\langle x,y\rangle = \langle x^2,y\rangle$. This proves the assertion with $x':=x^2$.

It can be verified that the same is true for $p=3$, even though the formulae for $x^{-2} y x^2$ and $[y,x^2]$ become slightly different:
\begin{align*}
 x^{-2} y x^2 
 &= e_{2,1}\inv (e_{1,1} e_{2,2} e_{3,1+2p} e_{4,1}) e_{2,1} 
   = e_{1,1} e_{2,1} e_{3,1+2p} e_{4,1+p}\ ,
 \\
 [y,x^2] 
 &= e_{4,1+p^2}\inv e_{3,1}\inv e_{3,1+2p} e_{4,1+p}
 \ .
\end{align*}
All conclusions, however, remain valid.

For $p=2$, take
\begin{align*}
x &:= e_{1,1} e_{2,1} e_{3,1} e_{4,1} = (1,9,5,13,3,11,7,15,2,10,6,14,4,12,8,16)
\ ,\\
x' &:= x^3 = (1,13,7,10,4,16,5,11,2,14,8,9,3,15,6,12)
\ ,\\
y &:=  e_{1,1} e_{3,4} e_{4,1} = (1,9,2,10)(3,11)(4,12)(5,15,7,13)(6,16,8,14)
\ ,
\end{align*}
then some computations show $\langle x,y\rangle=\langle x',y\rangle$, but $\ord([y,x])=2\neq 4=\ord([y,x'])$.
\end{proof}

\section{Results on strata of \texorpdfstring{$p$}{p}-origamis}\label{section: results on p-oris}

The goal of this section is to derive results about $p$-origamis from the results about $p$-groups in \Cref{section: results on p-groups}. In \Cref{subsection: strata of p-origamis}, we answer the question in which strata $p$-origamis occur. Subsequently, we study in \Cref{section: isomorphic deck transformation groups} under which conditions $p$-origamis with isomorphic deck groups lie in the same stratum.

\subsection{Strata of \texorpdfstring{$p$}{p}-origamis}\label{subsection: strata of p-origamis}

The answer to the question in which strata $p$-origamis occur depends on whether the considered prime is 2 or not. We begin this section with some facts that hold for all primes. In \Cref{section: strata of $2$-origamis} and \Cref{section: strata of $p$-origamis}, we consider $2$-origamis and $p$-origamis for odd primes $p$, respectively.

\begin{lemma}\label[lemma]{p-orgamis with abelian deck trafo gp}
	All normal origamis with abelian deck group lie in the stratum $\HH(0)$.
\end{lemma}

\begin{proof}
	Let $\mathcal{O}=(G,x,y)$ be a normal origami with abelian deck group $G$. Since $G$ is abelian, the commutator $[x,y]$ is trivial. By \Cref{connection order singularities and commutator}, the cover induced by the origami $\OO$ is unramified, i.e., $\OO$ lies in the stratum $\HH(0)$.
\end{proof}

In particular, the previous lemma shows that all $p$-origamis with abelian deck transformation group lie in the stratum $\HH(0)$.

Since all abelian groups are isoclinic, the following can be viewed as a generalization of the previous observation:

\begin{corollary} 
    The type of singularity of a normal origami depends only on the isoclinism class of its deck group. In particular, normal origamis with isoclinic deck groups of the same order lie in the same stratum.
\end{corollary}

\begin{proof} By \Cref{lem-isoclinic-groups}, the set of possible commutator orders for pairs of generators coincides for isoclinic $2$-generated groups. This determines the degree of the singularity. The group order together with this degree determines the stratum. 
\end{proof}

\begin{rem}\label[remark]{rem-stratum-maximal-class-2gps} As an example, we have seen that for each $n\geq 1$, there is a stratum containing all $2$-origamis with the dihedral group, the generalized quaternion group, or the semidihedral group of order $2^n$ as the group of deck transformations (see \Cref{cor-dihedral-groups}). It was computed to be $\HH(4\times \left( 2^{n-2}-1\right))$ in \Cref{iso 2-gp of max class same stratum}.
\end{rem}

\subsubsection{Strata of \texorpdfstring{$2$}{2}-origamis}\label{section: strata of $2$-origamis}

In this section, we classify the strata of $2$-origamis. We will see in \Cref{section: strata of $p$-origamis} that the occurring strata differ from the ones of $p$-origamis for odd primes $p$.

\begin{theorem}\label[theorem]{strata 2-origamis}
	Let $n\in \Z_{\geq0}$. For $2$-origamis of degree $2^n$, exactly the following strata appear
	\begin{enumerate}[\textbullet]
			\item $\HH(0),$
			\item $\mathcal{H}\left({2^{n-k}}\times\left(2^k-1\right)\right),$ where $1\le k\le n-2.$
	\end{enumerate}
\end{theorem}

\begin{proof}
	For $n\leq2$, all groups of order $2^n$ are abelian, so the corresponding $2$-origamis lie in the trivial stratum by \Cref{p-orgamis with abelian deck trafo gp}.
	
	Let $n,k\in \NN$ with $n>2$ and $k\le n-2$. By \Cref{constr of 2-gps}, there exists a $2$-generated group $G$ of order $2^n$ and generators $x,y$ such that $\ord([x,y])=2^k$. Hence the origami $(G,x,y)$ lies in the stratum $\HH(0)$ for $k=0$ and in $\HH\left({2^{n-k}}\times\left(2^k-1\right)\right)$ for $k>0$.
	
	It remains to prove that other strata cannot occur. Let $\OO=(G,x,y)$ be a $2$-origami of degree $2^n$. By \Cref{all singularities have same order - p group}, the only possible strata are of the form $\HH(k\times(a-1))$ where $a$ is the multiplicity of each singularity and $k$ is the number of singularities. Using \Cref{prop. exp(G')<n-1}, we deduce that the inequality $\exp(G')\le 2^{n-2}$ holds. By \Cref{connection order singularities and commutator}, the multiplicity of each singularity equals $\ord([x,y])$. Since $\ord([x,y])\le 2^{n-2},$ the claim follows.
\end{proof}

\begin{rem}\label[remark]{constr 2-oris gps}
    We recall the definition of the series of $2$-groups in the proof of \Cref{constr of 2-gps}. For $n,k\in \NN$ with $k\le n-2$, we define the semidirect product
	$$G_{(n,k)}^2 \vcentcolon=C_{2^{k+1}} \rtimes_\varphi C_{2^{n-k-1}}= \langle r,s ~|~ r^{2^{k+1}}=s^{2^{n-k-1}}=1,~s\inv rs=r\inv  \rangle.$$
    The commutator $[r,s]$ has order $2^k$ and thus the $2$-origami $(G_{(n,k)}^2 ,r,s)$ lies in the stratum $\HH\left({2^{n-k}}\times\left(2^k-1\right)\right)$.
    
    In particular, this shows that in each of the occurring strata there exists a $2$-origami with a semidirect product of two cyclic groups as deck transformation group.
\end{rem}

\begin{example}\label[example]{example-sl2-G2-3-1}
	For $n=3$ and $k=1$, we obtain the group
	$$G^2_{(3,1)}=C_{4} \rtimes_\varphi C_{2}= \langle r,s ~|~ r^{4}=s^{2}=1,~s\inv rs=r\inv  \rangle.$$
	This group is isomorphic to the dihedral group $D_8$. We consider the $2$-origami $\OO=(G^2_{(3,1)},r,s)$. The commutator $[r,s]$ has order $2$ and thus $\OO$ lies in $\HH(4\times 1)$.
	\begin{center} 
    	\captionsetup{type=figure}
	    \input{tikz/Beisp-G2_3-1-Zyl}
	    \caption{The $2$-origami $\OO=(G^2_{(3,1)},r,s)$ with marked horizontal cylinders.}
	\end{center}
	Choosing $(s,rs)$ as the pair of generators of $G^2_{(3,1)}$, we obtain the origami $\OO'=(G^2_{(3,1)},s,rs)$. Since $[s,rs]=r^{-2}$ has order 2, the origamis $\OO$ and $\OO'$ lie in the same stratum. In the following, we show that the origamis are different. We call a maximal collection of parallel closed geodesics on an origami a \textbf{cylinder}. For a normal origami $(H,x,y)$ the length of each cylinder in horizontal and vertical direction equals the order of $x$ and $y$, respectively. We conclude that the horizontal cylinders of $\OO'$ have length 2, whereas the horizontal cylinders of $\OO$ have length 4. Hence the origamis are different. (Alternatively, one can use  \Cref{origamis are equal} and the fact that group isomorphisms preserve orders.)
	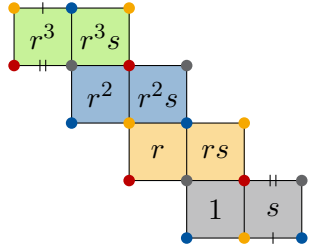
\begin{figure}[ht]
	    \centering
	    \input{tikz/Beisp-G2_3-1-v2-Zyl}
	    \caption{The $2$-origami $\OO'=(G^2_{(3,1)},s,rs)$ with marked horizontal cylinders. Opposite sides are identified unless marked otherwise.}
	\end{figure}
	
	The example above gives two different $2$-origamis with isomorphic deck group. In \Cref{section: isomorphic deck transformation groups}, we address the question whether the deck transformation group determines the stratum of a normal origami.
\end{example}

\begin{rem}	
    We recall that the matrix group $\SL(2,\Z)$ is generated by $S=\left(\begin{smallmatrix}0&-1\\1&0\end{smallmatrix}\right)$ and $T=\left(\begin{smallmatrix}1&1\\0&1\end{smallmatrix}\right)$. There is a natural action of $\SL(2,\Z)$ on origamis, where geometrically the action of $S$ corresponds to a rotation by $90^\circ$, while the action of $T$ corresponds to shearing the squares in the tiling of an origami (in the Euclidean plane) by $T$. One is interested in studying $\SL(2,\Z)$-orbits of origamis, because they are connected to the study of Teichm\"uller curves. 
    
    \begin{center}
        \captionsetup{type=figure}
        \input{tikz/shearing}
        \caption[Shearing]{The ``shearing'' action of $T^2=\left(\begin{smallmatrix}1&2\\0&1\end{smallmatrix}\right)\in\SL(2,\Z)$ on an origami. The opposite sides are identified.} 
        \label[figure]{fig:shearing}
    \end{center}        
    
    Alternatively, the action can be described with the help of the monodromy map. See \cite{WS-deficiency} for further information. For a normal origami $(H,x,y)$, the actions of $S$ and $T$ are given by
    $$    S.(H,x,y) = (H,y^{-1},x)
    ,\qquad
    T.(H,x,y) = (H,x,yx^{-1}).
    $$
    In this description, it can be verified easily, that the commutator of the generators $[x,y]$ is sent to a conjugated element in $H$ by the actions of $S$ and $T$, respectively. So in particular, to an element of the same order. Hence it is clear that the stratum is preserved. It follows that the set of normal origamis in a given stratum decomposes into orbits under the $\SL(2,\Z)$-action. This is also known from the geometric description.
    
    We note that \Cref{example-sl2-G2-3-1} shows an instance of such an $\SL(2,\Z)$-action because of the equation
    $$T.\OO'=\left(G,s,rss^{-1}\right)=(G,s,r)=\OO.$$
    In particular, $\OO$ and $\OO'$ lie in the same $\SL(2,\Z)$-orbit.
\end{rem}

\subsubsection{Strata of \texorpdfstring{$p$}{p}-origamis for odd primes \texorpdfstring{$p$}{p}}\label{section: strata of $p$-origamis}

Throughout this section, let $p$ denote an odd prime. We study the question in which strata $p$-origamis lie. Compared to the case of $2$-origamis fewer strata occur. This is shown in \Cref{strata p-origamis for p>2}. 

\begin{theorem}\label[theorem]{strata p-origamis for p>2}
	Let $n\in \Z_{\geq0}$. For $p$-origamis of degree $p^n$ exactly the following strata appear
	\begin{enumerate}
		\item[$\bullet$] $\HH(0)$
		\item[$\bullet$] $\mathcal{H}\left( {p^{n-k}} \times\left( p^k -1\right)\right) ,$ where $1\le k< \frac{n}{2}$.
	\end{enumerate}
\end{theorem}

\begin{proof}
    For $n\leq2$, all groups of order $p^n$ are abelian, so the corresponding $p$-origamis lie in the trivial stratum by \Cref{p-orgamis with abelian deck trafo gp}.
	
    Let $n,k\in \NN$ with $n>2$ and $k<\frac{n}{2}$. By \Cref{group theory: strata p-origamis for p>2}, there exists a $2$-generated group $G$ of order $p^n$ and generators $x,y$ such that $\ord([x,y])=p^k$. Hence the origami $(G,x,y)$ lies in the stratum $\HH(0)$ for $k=0$ and in $\HH\left({p^{n-k}}\times\left(p^k-1\right)\right)$ for $k>0$.
	It remains to prove that other strata cannot occur. Let $\OO=(G,x,y)$ be a $p$-origami of degree $p^n$. By \Cref{all singularities have same order - p group}, the only possible strata are of the form $\HH(k\times(a-1))$ where $a$ is the multiplicity of each singularity and $k$ is the number of singularities. By \Cref{connection order singularities and commutator}, the multiplicity of each singularity equals $\ord([x,y])$. Using \Cref{cor bound for ord xy}, we deduce that the inequality $\ord([x,y])\le p^{k}$ holds, where $k< \frac{n}{2}$.
\end{proof}

\begin{rem}\label[remark]{constr p-oris gps}
    We recall the definition of the series of $p$-groups in the proof of \Cref{group theory: strata p-origamis for p>2}. For $n,k\in \NN$ with $k<\frac{n}{2}$, we define the semidirect product
	$$G_{(n,k)}^p \vcentcolon=C_{p^{k+1}} \rtimes_\varphi C_{p^{n-k-1}}= \langle r,s ~|~ r^{p^{k+1}}=s^{p^{n-k-1}}=1,~s\inv rs=r^{p+1} \rangle.$$
    The commutator $[r,s]$ has order $p^k$ and thus the $p$-origami $(G_{(n,k)}^p ,r,s)$ lies in the stratum $\HH\left({p^{n-k}}\times\left(p^k-1\right)\right)$. As in \Cref{constr 2-oris gps}, this shows that in each of the occurring strata there exists a $p$-origami with a semidirect product of two cyclic groups as deck transformation group.
\end{rem}

\begin{example}
	For $p=3,~n=3$ and $k=1$, we obtain the group
	$$G^3_{(3,1)}=C_{9} \rtimes_\varphi C_{3}= \langle r,s ~|~ r^{9}=s^{3}=1,~s\inv rs=r^{4} \rangle.$$
	We consider the origami $\OO=(G^3_{(3,1)},r,s)$. The commutator $[r,s]$ has order $3$ and thus $\OO$ lies in $\HH(9\times2)$. Hence the origami has nine singularities of angle $3\cdot 2\pi$. 
    \begin{center}
        \captionsetup{type=figure}
	    \input{tikz/Beisp-G3_3-1-1}
	    \caption{The $3$-origami $(G^3_{(3,1)},r,s)$. All horizontal cylinders are of length $9$ and all vertical cylinders are of length $3$.}
	\end{center}
\end{example}

\subsection{\texorpdfstring{$p$}{p}-origamis with isomorphic deck groups}\label{section: isomorphic deck transformation groups}

In this section, we study the question whether the deck transformation group determines the stratum of a normal origami. This question was motivated by computer experiments. For $p$-origamis, computer experiments suggested that the stratum depends only on the isomorphism class of the deck transformation group. Using \Cref{ex. iso deck groups not same stratum}, we show that this does not hold for all finite groups.

\begin{example}
    For $n\in \mathbb{N}_{\ge5}$, we consider the alternating group $A_n$ with the following pairs of generators $((1,2,\dots,n-1,n),(1,2,3))$ and $((3,4,\dots, n-1,n),(1,3)(2,4))$. Recall from \Cref{ex. iso deck groups not same stratum} that the order of the commutators $[(1,2,\dots,n-1,n),(1,2,3)]$ and $[((3,4,\dots, n-1,n),(1,3)(2,4))]$ are $3$ and $5$, respectively. Hence the normal origami $\OO_n\vcentcolon=(A_n, (1,2,\dots,n-1,n),(1,2,3))$ has $\frac{n!}{6}$ singularities of multiplicity $3$, whereas the normal origami $\OO_n'\vcentcolon= (A_n,(3,4,\dots, n-1,n),(1,3)(2,4))$ has $\frac{n!}{10}$ singularities of multiplicity $5$. It follows that there are two pairs of generators of $A_n$ defining normal origamis lying in different strata.
    
    Recall that we multiply permutations from the left because we label the squares of a normal origami by multiplying generators of the deck group from the right.
    
    The origami constructed in {\cite[Example 7.3]{Athreya2018}} is a normal origami with deck group $A_5$. It lies in the same stratum as origami $\OO_5'$ and thus could replace $\OO_5'$ in the example above for $n=5$. 
\end{example}

Recall that two normal origamis $(G,x_1,y_1)$ and $(G,x_2,y_2)$ with isomorphic group of deck transformation group lie in the same stratum if and only if the orders of the commutators $[x_1,y_1]$ and $[x_2,y_2]$ agree. This is is the case for all possible pairs of generators of a group $G$ if and only if the deck transformation group has \property (see \Cref{def property C}). Using \Cref{thm-diagram} and \Cref{rem-small-groups}, we obtain \Cref{thm-B} from the introduction.

\begin{theorem}\label[theorem]{thm-deck-group-determines-stratum}
    Let $G$ be a finite $2$-generated $p$-group. If $G$ satisfies one of the properties (1) to (8), then all $p$-origamis with deck group $G$ lie in the same stratum.
    \begin{enumerate}
    	\begin{minipage}{7cm}
			\item $G$ is regular.
            \item $G$ has maximal class.
            \item $G$ is powerful.
            \item $G'$ is regular.        
		\end{minipage}
		\begin{minipage}{8cm}
            \item $G'$ is powerful.
            \item $G'$ is order-closed.
		    \item $G$ has order at most $p^{p+2}$. 
		    \item $G$ has nilpotency class at most $p$.
		\end{minipage}
    \end{enumerate}
\end{theorem}

\begin{proof}
    By definition, a finite $2$-generated group $G$ has \property, if there exists a natural number $n$ such that for each $2$-generating set $\{x,y\}$ of $G$ the order of $[x,y]$ equals $n$. Hence it is sufficient to show that \property holds for all groups satisfying one of the properties (1) to (8). This follows from \Cref{thm-diagram}. The connection to groups up to a certain order or nilpotency class is made in \Cref{rem-small-groups}.
\end{proof}

For certain $p$-groups with \property, we studied the constant given by the order of the commutator of a pair of generators in \Cref{section: gp thy iso deck groups} (see~\Cref{iso 2-gp of max class same stratum}, \Cref{lem maximal class p^n 5<=n<=p+1}, and \Cref{minimal non-P_i gps}). We deduce the corresponding results for the strata of the respective $p$-origamis.

\begin{corollary}
    Any $2$-origami of degree $2^n$ whose deck group has maximal class lies in the stratum $\HH\left(4\times \left(2^{n-2}-1\right)\right)$.
\end{corollary}

\begin{corollary}
    For $5\le n\le {p+1}$ and an odd prime $p$, any $p$-origami of degree $p^n$ whose deck group has maximal class lies in the stratum $\HH\left( p^{n-1}\times\left( p-1\right)\right)$.
\end{corollary}

\begin{corollary}
    Any $p$-origami of degree $p^n$ whose deck transformation group is a minimal non-power-closed $p$-group or a minimal non-order-closed $p$-group lies in the stratum $\HH\left( p^{n-1}\times\left( p-1\right)\right)$.
\end{corollary}

In \Cref{gp without property (C)}, we constructed for each prime a $2$-generated $p$-group that does not have \property. Hence we obtain the following proposition.

\begin{proposition}
    For each prime $p$, there exist $p$-origamis with isomorphic deck transformation group that lie in different strata.
\end{proposition}

\begin{proof}
    In \Cref{gp without property (C)}, we proved for each prime $p$ the existence of a $2$-generated $p$-group $H_p$ which is contained in the Sylow $p$-subgroup of the symmetric group $S_{p^4}$ and does not have \property. Hence there exist $p$-origamis with deck transformation group  isomorphic to $H_p$ that lie in different strata.
\end{proof}

\begin{rem} 
    The 2-group $G$ with generating sets $(x,y)$ and $(x, y^3)$ defined in \Cref{ex. power-closed gp not property C} is weakly power-closed, but not power-closed. Recall that the orders of the commutators $\ord([x,y])$ and $\ord([x,y^3])$ are $4$ and $2$, respectively. We obtain that the $2$-origamis $(G,x,y)$ and $(G,x,y^3)$ lie in different strata, namely $\HH(2^{10}\times 3)$ and $\HH(2^{11}\times 1)$. Here we use that the group $G$ has order $2^{12}$. 
\end{rem}

\section{Outlook: Infinite origamis and pro-\texorpdfstring{$p$}{p} groups}\label{section infinite origamis}

So far, we have considered surfaces that are also called finite translation surfaces, i.e., the surface can be described as finitely many polygons with edge identifications via translations. As a generalization, \textbf{infinite translation surfaces} have been studied during the past 10 years (see e.g. \cite{bowman-valdez}). In contrast to finite translation surfaces, one allows countably many polygons glued by translations. For a detailed introduction to infinite translation surfaces see \cite{Randecker} and \cite{DHV}. 

In this section, we consider a well-known infinite translation surface called staircase origami. Moreover, we generalize the notion of \property to pro-$p$ groups, certain infinite analogs of finite $p$-groups. We then transfer some results from \Cref{section: gp thy iso deck groups} to pro-$p$ groups and deduce conclusions about a class of translation surfaces which we call infinite normal origamis.

Let $\OO\to \T$ be a countably infinite, normal cover of the torus $\T$ ramified over at most one point. Then $\OO$ is called an \textbf{infinite normal origami}. Amongst others, these surfaces have been studied by \cite{Karg}, where they are called regular origamis. As in the finite case, they correspond to a special class of infinite translation surfaces where all polygons are squares of the same size. The concepts introduced in \Cref{section: geometric motivation} carry over to infinite origamis. Given a countably infinite group $G$ with $2$-generating set $(x,y)$, one constructs an infinite normal origami $(G,x,y)$ as in \Cref{p-ori = 2-gen set}. One has a natural bijection between the squares in the tiling and the elements of the deck group. Singularities of infinite normal origamis can have finite cone angle, i.e., $2\pi n$ for $n\in\N$ as in the case of finite origamis, or infinite angle, i.e., a neighborhood of the singularity is isometric to a neighborhood of the branching point of the infinite cyclic branched cover of $\R^2$. As in the case of finite normal origamis, the cone angle of all singularities of the origami $(G,x,y)$ coincide and are equal to the order of the commutator $[x,y]$. 

\begin{example}\label[example]{ex. staircase origami}
    An example for an infinite normal origami is the staircase origami $St_\infty$ in \Cref{staircase origami}. It has been studied both from the geometric (see e.g. \cite{HubertSchmithuesen}) and from the dynamical point of view (see e.g. \cite{HooperWeiss}). 
    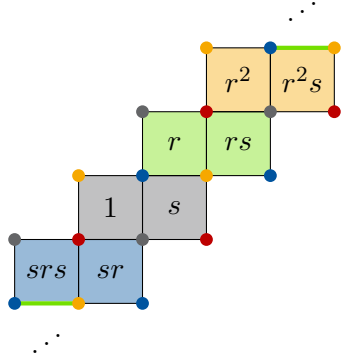
\begin{figure}
        \centering
        \input{tikz/staircase-origami}
        \caption{The infinite staircase origami $St_\infty=(D_\infty,s,sr)$ has four singularities of infinite cone angle. All vertical and horizontal cylinders have length 2. Opposite sides are identified.}\label[figure]{staircase origami}
    \end{figure}
    The deck group of the origami $St_\infty$ is the infinite dihedral group
    $$D_\infty\vcentcolon=\langle r,s~|~s^2=1, srs=r^{-1}\rangle,$$
    for $St_\infty$, the elements $s,sr$ are chosen as generators.
    The commutator subgroup of $D_\infty$ is the infinite cyclic group generated by $[r,s]=r^{-2}$. Hence, for any pair of generators $x,y$, the commutator $[x,y]$ has order infinity. We conclude that each infinite normal origami with deck group $D_\infty$ has $4$ singularities of infinite cone angle.
    
    Choosing the generators $r,s$ we obtain a surface different from $St_\infty$. In contrast to the latter, it has $2$ infinite horizontal cylinders, as shown in \Cref{origami 2 D_infty}.
    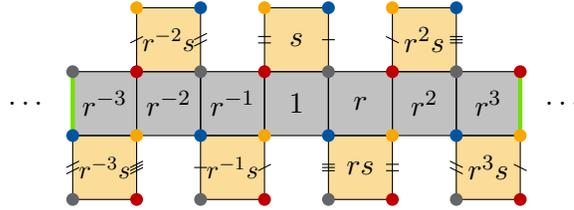
\begin{figure} 
        \centering
        \input{tikz/ori2-D_infty}
        \caption{The origami $(D_\infty,r,s)$ has $2$ infinite horizontal cylinders and infinitely many vertical cylinders of height 2. Opposite sides are identified unless marked otherwise.}\label[figure]{origami 2 D_infty}
    \end{figure}
\end{example}

The infinite dihedral group is a dense subgroup of the pro-$2$ group $\Z_{(2)}\ltimes C_2$. Pro-$p$ groups have played an essential role in the study of finite $p$-groups (see e.g. \cite{LGMcKpro-p}). It is a natural question, whether results of \Cref{section: gp thy iso deck groups} and \Cref{section: isomorphic deck transformation groups} can be transferred to certain infinite groups, in particular profinite and pro-$p$ groups, and infinite normal origamis. To this end, we extend the definition of \property to possibly infinite $2$-generated groups. Note that, as we do not consider topological groups yet, the groups under consideration are still countable.

\begin{definition} A (possibly infinite) $2$-generated group $G$ has \textbf{\property} if there is an element $k\in\N\cup\{\infty\}$ such that the order of $[x,y]$ equals $k$ for any pair $x,y$ of generators of $G$.
\end{definition}

\begin{example} Recall that the infinite dihedral group $D_\infty$ has an infinite cyclic commutator subgroup, i.e., $D_\infty$ has \property (see \Cref{ex. staircase origami}).
\end{example}

\newcommand\Ghat{\Hat{G}}

In the following, we consider $2$-generated profinite groups. For this purpose, recall that an \textbf{inverse system} of groups is a collection of groups $(G_i)_{i\in I}$ indexed by a directed poset $I$ and homomorphisms $\psi_{i,j}\colon G_j\to G_i$ for all $i,j\in I, i\leq j,$ such that $\psi_{i,i}$ is the identity and $\psi_{i,j} \circ \psi_{j,k}=\psi_{i,k}$ for all $i\leq j\leq k$. The \textbf{inverse limit} of an inverse system $(G_i)_i ,(\psi_{i,j})_{i\leq j}$ is the group 
$$
\Ghat := \{ (g_i)_i \in \prod_i G_i : \psi_{i,j}(g_j) = g_i \ \forall i\leq j\}
$$
together with the projection maps $\pi_i\colon \Ghat\to G_i$ onto the components. It is the categorical limit for the diagram described by the inverse system in the category of groups, and is distinguished by the corresponding universal property (see Diagram (\ref{daigram inverse limit})).

Recall that a \textbf{profinite group} is a Hausdorff, compact, completely-disconnected topological group. For any profinite group $G$, the quotients $G/N$ for all normal open subgroups $N$ of $G$ form an inverse system of finite groups whose inverse limit is isomorphic to $G$, and any inverse limit of an inverse system of finite groups meets the definition of a profinite group above. So equivalently, one can define profinite groups as those which are isomorphic to the inverse limit of an inverse system of finite groups.

For profinite, or more generally, topological groups, we can consider \textbf{topological generating sets}, those subsets which generate a dense subgroup. A profinite group is topologically $2$-generated if and only if it is isomorphic to the inverse limit of an inverse system of $2$-generated finite groups (\cite[Prop.~1.5]{DSMS}).

This allows us to define:

\begin{definition} A topologically $2$-generated profinite group has \textbf{\property} if it is isomorphic to the inverse limit of an inverse system of finite groups which have \property.
\end{definition}

\begin{lemma}\label[lemma]{lemma dense subgps}
    Any $2$-generated dense subgroup $G$ of a topologically $2$-generated profinite group $\Ghat$ with \property has \property. Moreover, the commutator order of a pair of generators is the same for any choice of a $2$-generated dense subgroup $G$.
\end{lemma}

\begin{proof} Let us write $\Ghat\cong \underleftarrow{\lim} (G_i,\psi_{i,j})$ for an inverse system $(G_i)_i,(\psi_{i,j})_{i\leq j}$ of finite groups with \property, and let us denote the structural projections of the inverse limit by $\pi_i\colon \Ghat\to G_i$.

From the inverse limit construction it is clear that the order of any element $g\in\Ghat$ is the supremum, finite or infinite, of the element orders $(\ord(\pi_i(g)))_i$ in the respective groups $G_i$. If $g$ is the commutator of a pair of generators $x,y$ of the dense subgroup $G$, then $\pi_i(g)$ is the commutator of a pair of generators in $G_i$ for every $i$, since $\pi_i$ is an epimorphism. 
But as $G_i$ has \property for any $i$, the number $\ord(\pi_i(g))$ is a constant for each $i$, independent of the choices of $G$ and its generators $x,y$.
\end{proof}

\begin{corollary}
    Let $G$ be a topologically $2$-generated profinite group with \property. The singularities of all infinite normal origamis whose deck group is a dense subgroup of $G$ have the same cone angle. 
\end{corollary}

\begin{rem}\label[remark]{constr infinite ori from finite ones}
    Let $\left(G_i)_i,(\psi_{i,j}\right)_{i\le j}$ be an inverse system of $2$-generated finite groups with compatible $2$-generating sets $(x_i,y_i)$, i.e., $\psi_{i,j}(x_j)=x_i$ and $\psi_{i,j}(y_j)=y_i$. Denote the inverse limit by $(\hat G, \pi_i)$. Applying the universal mapping property of the inverse limit to the free group $F_2=\langle a,b \rangle$ and the monodromy maps $m_i:F_2\to G_i$, one obtains a unique group homomorphism $\alpha: F_2\to\hat G$ and the following commutative diagram
    \begin{align}\label{daigram inverse limit}
        \begin{split}
	        \xymatrix{
		        &F_2 \ar@{-->}[d]_-{\alpha} \ar@/^1.0pc/[rdd]^{m_j} \ar@/_1.0pc/[ldd]_{m_i}& \\
	            &\Hat{G} \ar[dl]_-{\pi_i} \ar[rd]^-{\pi_{j}}&\\
		        G_j \ar[rr]^-{\psi_{i,j}}& & G_i.
	        }
	    \end{split}
	\end{align}
	The image $\alpha(F_2)$ is a dense subgroup of $\hat G$ with $2$-generating set (in the sense of classical group theory) $x\vcentcolon=\alpha(a), y\vcentcolon=\alpha(b)$. Note that $\pi_i(x)=x_i, \pi_i(y)=y_i$ for all $i\in I$. One obtains an infinite normal origami $(\alpha(F_2),x,y)$ associated with the infinite sequence of finite normal origamis $(G_i,x_i,y_i)$. 
	
	Similarly, any given infinite normal origami $(H,x',y')$ for a dense subgroup $H$ of $\hat G$ yields infinitely many finite normal origamis $(G_i,\pi_i(x'_i),\pi_i(y'_i))$. Here, we use that the image $\pi_i(H)$ equals $G_i$ for each $i\in I$ (see \cite[Proposition 1.5]{DSMS}). 
\end{rem}

Recall that for any prime $p$, a \textbf{pro-$p$ group} is a profinite group $G$ such that $G/N$ is a finite $p$-group for any open normal subgroup $N$ of $G$. Equivalently, it is a group which is isomorphic to the inverse limit of an inverse system of finite $p$-groups. Pro-$p$ groups play a central role in the coclass conjectures by Leedham-Green and Newman (\cite{LGN-coclass}), proved by Leedham-Green (\cite{LG-coclass}) and Shalev (\cite{Sh-coclass}), concerning a way of classifying all finite $p$-groups.

\newcommand\cl{\operatorname{cl}}
We can now extend some of our results from finite $p$-groups to pro-$p$ groups. Let us call a pro-$p$ group \textbf{weakly order-closed} if products of elements of order $p^k$ have order at most $p^k$, for any $k\geq 0$. Let us also recall that a pro-$p$ group is called \textbf{powerful} if $G/\cl(\{g^{p^k} : g\in G\})$ is abelian for $k=1$ if $p>2$ and for $k=2$ if $p=2$, where $\cl(S)$ denotes the minimal closed subgroup generated by a set $S$ (see \cite[Definition 3.1]{DSMS}).

We observe that key results as \Cref{G' weakly order-closed implies property} or \Cref{iso powerful p-gp same stratum} for finite $p$-groups can be generalized to pro-$p$ groups.

\begin{lemma} A topologically $2$-generated pro-$p$ group $G$ has \property if either \begin{itemize}
    \item[(1)] $G'$ is weakly order-closed, or
    \item[(2)] $G'$ is powerful.
\end{itemize}
\end{lemma}

\begin{proof} Let $x,y$ be generators of $G$. The commutator subgroup $G'$ is generated by conjugates of $[x,y]$, all having the same order, say $p^k$ for $k\geq 1$. We will show that assuming (1) or (2) implies this order is, in fact, the exponent of $G'$, independent of the choice of $x,y$. This proves the assertion that $G$ has \property.

If we assume (1), then the countable subgroup of $G'$ generated (without topological closure) by all conjugates of $[x,y]$ consists of elements of order at most $p^k$. Hence, indeed, $p^k$ is the exponent of the countable subgroup, and of $G'$ being its closure. 

If we assume (2), then $G'$ is a powerful pro-$p$ group generated by conjugates of $[x,y]$. Hence, by \cite[Prop.~3.6~(iii)]{DSMS}, the set of all $p^k$-th powers in $G'$ equals the closed subgroup generated by the $p^k$-th powers of the conjugates of $[x,y]$, which is the trivial subgroup. Thus, all $p^k$-th powers in $G'$ have to be trivial, so $p^k$ is the exponent of $G'$.
\end{proof}

As for finite $p$-groups this has implications for families of normal origamis:

\begin{corollary}
    Let $G$ be a topologically $2$-generated pro-$p$ group which is either powerful or has a weakly order-closed commutator subgroup. The singularities of all infinite normal origamis whose deck group is a dense subgroup of $G$ have the same cone angle. 
\end{corollary}

We conclude our exploration into the world of infinite deck groups with some examples. 

\begin{example}
    In this example, we introduce a setup to construct inverse systems from semidirect products of cyclic groups. Such groups appeared several times in \Cref{section: gp thy for strata results} and \Cref{section: results on p-oris}. For a prime $p$, let $C_{p^m}=\langle x\rangle$ and $C_{p^\ell}=\langle y\rangle$ be cyclic groups of order $p^m$ and $p^\ell$, respectively. If $a\in \Z$ is coprime to $p$ and $a^{p^m}\equiv 1 \mod p^\ell$, then $\varphi_a: C_{p^m}\to \aut(C_{p^\ell})$, $x\mapsto (y\mapsto y^a)$, defines a semidirect product $$H_{(\ell,m,a)}:=C_{p^\ell}\rtimes_{\varphi_a} C_{p^m}=\langle x,y ~|~x^{p^m}=y^{p^\ell}=1,x^{-1}yx=y^a\rangle.$$ 
    
    Fixing $a\in \Z$ coprime to $p$ and $m',\ell' \in \N$ such that $a^{p^{m'}}\equiv 1 \mod p^{\ell'}$, we construct two different inverse systems of finite $p$-groups. For $m\ge m'$, the semidirect product $H_{(\ell',m,a)}$ is well defined and we obtain epimorphisms $H_{(\ell',m+1,a)}\to H_{(\ell',m,a)}$ sending $x\mapsto x$ and $y\mapsto y$ in the respective groups. This defines an inverse system with a semidirect product $C_{p^\ell}\rtimes \Z_{(p)}$ as inverse limit.
    
    For $m\ge m'$, $\ell\ge \ell'$ with $m-\ell=m'-\ell'$, the semidirect product $H_{(\ell,m,a)}$ is well defined and we obtain epimorphisms $H_{(\ell+1,m+1,a)}\to H_{(\ell,m,a)}$ sending $x\mapsto x$ and $y\mapsto y$ in the respective groups. This defines an inverse system with a semidirect product $\Z_{(p)}\rtimes \Z_{(p)}$ as inverse limit.

    Choosing compatible $2$-generating sets for the groups forming an inverse system, one obtains an infinite sequence of finite normal origamis and an infinite normal origami which has a dense subgroup of the inverse limit as its deck group (see \Cref{constr infinite ori from finite ones}).
    
    Note, that all groups of the form $H_{(\ell,m,a)}$ have \property, since the commutator subgroup is always cyclic (generated by $y^{a-1}$). Hence, both constructions of inverse systems yield pro-$p$ groups with \property. The constructions can be applied to the groups $G^{(p)}_{(n,k)}$ constructed in \Cref{constr 2-oris gps} and \Cref{constr p-oris gps} for $p=2$ and $p>2$, respectively. 
    
    The dihedral groups $D_{2^n}=\langle r_n,s_n~|~r_n^{2^{n-1}}=s_n^2=1, s_n r_n s_n=r_n^{-1}\rangle$, considered in \Cref{ex. staircase origami}, form an inverse system constructed in a similar, but slightly different way. Here, one chooses $m=1, a=-1$ and lets $\ell$ vary. The infinite dihedral group $D_\infty$ is a $2$-generated countable dense subgroup of the inverse limit $\hat D_2=\Z_{(2)}\rtimes C_2$ of the $2$-generated $2$-groups $\left(D_{2^n}\right)_{n\geq0}$. For $n\in \N$ the tuples $(r_n,s_n)$ form compatible $2$-generating sets. Using the construction in \Cref{constr infinite ori from finite ones} we obtain the normal infinite origami $(D_\infty,r,s)$ in \Cref{origami 2 D_infty}. Choosing $(s_n,s_n r_n)$ as compatible $2$-generators the construction yields the infinite staircase origami in \Cref{staircase origami}.
\end{example}

In the following example, we construct an inverse system taking the quaternion group as a starting point. In this way we obtain an infinite series of normal origamis covering the Eierlegende Wollmilchsau (see \Cref{eierlegende wollmilchsau}).

\begin{example}\label[example]{expl::generalized-wollmilchsau} 
The groups
$$
W_n = \langle x,y \mid x^{2^{n+1}} = y^{2^{n+1}} = x^{2^n} y^{2^n} = 1, x\inv y x = y\inv \rangle
$$
are $2$-generated $2$-groups of order $2^{2n+1}$. For any $n\geq 1$, the element $z:=y^2=[x,y]$ is in the commutator subgroup $W'_n$, it clearly commutes with $y$, and
$$
x\inv z x = x\inv y^2 x = y^{-2} = z\inv
\ .
$$
Hence, by \Cref{commutator-generated}, the commutator subgroup $W'_n=\langle z\rangle$ is cyclic of order $2^n$. The defining relations of $W_n$ imply that its elements are of the form $x^a y^b$ for $0\leq a< 2^{n+1}$ and $0\leq b<2^n$. The images of $x, y$ in the abelianization $W_n/W'_n=W_n/\langle y^2\rangle$ have order $2^n$ and $2$, respectively, and the abelianization is isomorphic to $C_{2^n}\times C_2$.

We have epimorphisms from $W_{n+1}$ to $W_n$ for any $n\geq1$ sending $x\mapsto x$, $y\mapsto y$ (and $z\mapsto z$) in the respective groups. $W_1$ is isomorphic to the quaternion group with $8$ elements. 
The generators $(x,y)$ (viewed as elements in the groups $W_n$) form a set of compatible generators. Hence, we obtain an infinite sequence of normal origamis $(W_n,x,y)$ covering the Eierlegende Wollmilchsau. We give an example of these origamis for $n=2$ in \Cref{fig::generalized-wollmilchsau}.

\begin{center}
    \captionsetup{type=figure}
    \input{tikz/generalized-Wollmilchsau}
    \caption{The origami $(W_2,x,y)$ as in \Cref{expl::generalized-wollmilchsau} is a cover of the Eierlegende Wollmilchsau of degree $4$.}
    \label{fig::generalized-wollmilchsau}
\end{center}
\end{example}

\appendix

\section{\texorpdfstring{\texttt{GAP}}{GAP} code} \label{sec-code}

All of the following is \texttt{GAP4} code (see \cite{GAP4}).

\begin{lstlisting}[label=code-1,caption={For natural numbers $n,k$ with $1\le k\le n-2$, the following code defines the group $G_{(n,k)}^2$ as well as the corresponding $2$-generating set constructed in \Cref{constr of 2-gps}.\medskip}] 
G2 := function(n,k)
local C1, C2, alpha, phi, G, x, y;

C1 := CyclicGroup(2^(k+1)); 
C2 := CyclicGroup(2^(n-k-1));
alpha := GroupHomomorphismByImages(C1, C1, [C1.1], [(C1.1)^(-1)]);
phi := GroupHomomorphismByImages(C2, AutomorphismGroup(C1), [C2.1],
       [alpha]);

G := SemidirectProduct(C2, phi,C1);
x := Image(Embedding(G,2), C1.1); 
y := Image(Embedding(G,1), C2.1);

return [G, x, y];
end;
\end{lstlisting}

\begin{lstlisting}[label=code-2,caption={For an odd prime $p$ and natural numbers $n,k$ with $k<\frac{n}{2}$, the following code defines the group $G_{(n,k)}^p$ as well as the corresponding $2$-generating set constructed in \Cref{group theory: strata p-origamis for p>2}.\medskip}] 
Gp := function(p,n,k)
local C1, C2, alpha, phi, G, x, y;

C1 := CyclicGroup(p^(k+1));
C2 := CyclicGroup(p^(n-k-1));
alpha := GroupHomomorphismByImages(C1, C1,[C1.1], [(C1.1)^(p+1)]);
phi := GroupHomomorphismByImages(C2, AutomorphismGroup(C1), [C2.1],
       [alpha]);

G := SemidirectProduct(C2, phi, C1);
x := Image(Embedding(G,2), C1.1);
y := Image(Embedding(G,1), C2.1);

return [G, x, y];
end;
\end{lstlisting}

\begin{lstlisting}[label=code-counterex,caption={Variations of the following code were used to find $p$-groups which do not have \property.\medskip}] 
p := 3; n := 4;
g := SylowSubgroup(SymmetricGroup(p^n), p);

repeat x := Random(g); y := Random(g);
until Order(Comm(x, y)) <> Order(Comm(x, y^2));
\end{lstlisting}

\begin{lstlisting}[label=code-powerclosed,caption={The following code defines a function to test whether a given $p$-group is weakly power-closed (i.e., products of $p^k$-th powers are $p^k$-th powers for any $k\geq 0$), and uses it to find a $2$-generated subgroup $G$ of the $2$-Sylow subgroup of the symmetric group $S_{2^4}$ with generators $x,y$ such that {$\ord([x,y])\neq\ord([x,y^3])$} and $G'$ is weakly power-closed.\medskip}]
IsWeaklyPowerClosedPGroup := function(g)
  local powers, el;
  if IsTrivial(g) then return true; fi;
  powers := g;
  repeat
    powers := Set(powers, x -> x^PrimePGroup(g));
    if Size(Group(powers)) > Size(powers) then return false; fi;
  until IsTrivial(powers);
  return true;
end;

p := 2; n := 4;
gg := SylowSubgroup(SymmetricGroup(p^n), p);;
repeat x := Random(gg); y := Random(gg); 
  g := Group(x, y); d := DerivedSubgroup(g);
until Order(Comm(x, y)) <> Order(Comm(x, y^(p+1)))
  and IsWeaklyPowerClosedPGroup(d);
\end{lstlisting}


\bibliographystyle{amsalpha}
\bibliography{biblio}

\end{document}

%% file: tikz/colors.tex
\definecolor{orange}{HTML}{f6a800} 
\definecolor{red}{HTML}{bd0000} 
\definecolor{reddot}{HTML}{bd0000} 
\definecolor{green}{HTML}{74DF00} 
\definecolor{gray}{HTML}{646567} 
\definecolor{darkblue}{HTML}{00549F}
\definecolor{lightblue}{HTML}{e8f1fa}

\definecolor{gray2}{HTML}{9c9e9f} 
\definecolor{lightblue2}{HTML}{8ebae5}
\definecolor{petrol}{HTML}{2d7f83}
\definecolor{lila}{HTML}{6A0888}

\definecolor{backgroundcolor}{HTML}{FFFFFF}
\definecolor{backgroundcolor1}{HTML}{e8f1fa}
\definecolor{backgroundcolor2}{HTML}{dcdddd}
\definecolor{backgroundcolor3}{HTML}{eceded}

\tikzstyle{greenline}=[color=green]
\tikzstyle{orangeline}=[color=orange]
\tikzstyle{redline}=[color=red]
\tikzstyle{whiteline}=[color=white]

\tikzstyle{orangedot}=[fill=orange]

\ifbw 

\tikzstyle{green}=[draw=black,fill=black!50]
\tikzstyle{darkblue}=[draw=black,fill=white]
\tikzstyle{orange}=[draw=black,fill=black!85, preaction={fill, white}]
\tikzstyle{gray}=[draw=black,fill=black!20]
\tikzstyle{red}=[draw=black,fill=black!85]

\tikzstyle{petrol}=[fill=black,draw=white,preaction={draw=black,line width=3pt}]
\tikzstyle{lila}=[fill=white,draw=black,copy shadow={shadow scale=1.5,shadow xshift=0pt, shadow yshift=0pt}]
\tikzstyle{lightblue2}=[fill=black!30,draw=black,preaction={draw=black!30,line width=3pt}]
\tikzstyle{gray2}=[fill=black!20,draw=black!60,preaction={draw=black!60,line width=3pt}]

\tikzstyle{orangedot}=[fill=black!50,draw=white,preaction={draw=black!50,line width=3pt}]

\tikzstyle{greenline}=[draw=black!20,dashed]
\tikzstyle{orangeline}=[draw=black!20,dotted]
\tikzstyle{redline}=[draw=black!20,solid]

\fi

%% file: tikz/wollmilchsau.tex
\begin{tikzpicture}[remember picture,scale=0.9]

\node[scale=0.9] 
(1) at (-0.5,0.5) {$1$};
\node[scale=0.9] 
($i$) at (0.5,0.5) {$i$};
\node[scale=0.9] 
($-1$) at (1.5,0.5) {$-1$};
\node[scale=0.9] 
($-i$) at (2.5,0.5) {$-i$};

\node[scale=0.9] 
(j) at (-0.5,1.5) {$j$};
\node[scale=0.9] 
(-k) at (0.5,-0.5) {$-k$};
\node[scale=0.9] 
($-k$) at (2.5,-0.5) {$k$};
\node[scale=0.9] 
($-j$) at (1.5,1.5) {$-j$};


\path[draw] (0,0) -- (1,0) -- (1,1) -- (0,1) -- (0,0);
\path[draw] (1,0) -- (2,0) -- (2,1) -- (1,1);
\path[draw] (2,0) -- (3,0) -- (3,1) -- (2,1);
\path[draw] (0,0) -- (-1,0) -- (-1,1) -- (0,1);

\path[draw] (0,0) -- (0,-1) -- (1,-1) -- (1,0);
\path[draw] (2,0) -- (2,-1) -- (3,-1) -- (3,0);

\path[draw] (1,1) -- (1,2) -- (2,2) -- (2,1);
\path[draw] (0,1) -- (0,2) -- (-1,2) -- (-1,1);

\fill[darkblue] (0,0) circle (0.1);
\fill[darkblue] (2,0) circle (0.1);
\fill[darkblue] (2,2) circle (0.1);
\fill[darkblue] (0,2) circle (0.1);

\fill[orangedot] (1,0) circle (0.1);
\fill[orangedot] (3,0) circle (0.1);
\fill[orangedot] (1,2) circle (0.1);
\fill[orangedot] (-1,2) circle (0.1);
\fill[orangedot] (-1,0) circle (0.1);

\fill[gray] (0,1) circle (0.1);
\fill[gray] (2,1) circle (0.1);
\fill[gray] (0,-1) circle (0.1);
\fill[gray] (2,-1) circle (0.1);

\fill[red] (1,-1) circle (0.1);
\fill[red] (3,-1) circle (0.1);
\fill[red] (1,1) circle (0.1);
\fill[red] (3,1) circle (0.1);
\fill[red] (-1,1) circle (0.1);


\path[draw] (-1.1,0.5) -- (-0.9,0.5);
\path[draw] (2.9,0.5) -- (3.1,0.5);

\path[draw] (-0.1,-0.45) -- (0.1,-0.45);
\path[draw] (-0.1,-0.55) -- (0.1,-0.55);

\path[draw] (-.1,1.45) -- (.1,1.45);
\path[draw] (-.1,1.55) -- (.1,1.55);

\path[draw] (0.9,-0.5) -- (1.1,-0.5);
\path[draw] (0.9,-0.4) -- (1.1,-0.4);
\path[draw] (0.9,-0.6) -- (1.1,-0.6);
\path[draw] (0.9,1.5) -- (1.1,1.5);
\path[draw] (0.9,1.4) -- (1.1,1.4);
\path[draw] (0.9,1.6) -- (1.1,1.6);

\path[draw] (1.9,1.6) -- (2.1,1.5) -- (1.9,1.4);

\path[draw] (1.9,-.6) -- (2.1,-.5) -- (1.9,-.4);

\path[draw] (-.9,1.6) -- (-1.1,1.5) -- (-.9,1.4);

\path[draw] (3.1,-.6) -- (2.9,-.5) -- (3.1,-.4);


\path[draw] (-0.55,1.9) -- (-0.45,2.1);
\path[draw] (1.55,.1) -- (1.45,-.1);

\path[draw] (-0.6,-.1) -- (-0.5,.1);
\path[draw] (-0.5,-.1) -- (-0.4,.1);

\path[draw] (1.4,1.9) -- (1.5,2.1);
\path[draw] (1.5,1.9) -- (1.6,2.1);

\path[draw] (.45,1.1) -- (.55,.9);
\path[draw] (2.45,-.9) -- (2.55,-1.1);

\path[draw] (2.4,1.1) -- (2.5,.9);
\path[draw] (2.5,1.1) -- (2.6,.9);

\path[draw] (.4,-.9) -- (.5,-1.1);
\path[draw] (.5,-.9) -- (.6,-1.1);
\end{tikzpicture}

%% file: tikz/commutator-1.tex
\begin{tikzpicture}[remember picture,scale=1.6]


\node[scale=0.9] 
(1) at (-0.5,0.5) {$x^{-1}y^{-1}$};
\node[scale=0.9] 
($i$) at (0.5,0.5) {$x^{-1}y^{-1}x$};
\node[scale=0.9] 
($i$) at (0.5,1.5) {$1$};
\node[scale=0.9] 
($i$) at (2,1.5) {$[x,y]$};
\node[scale=0.9] 
(j) at (-0.5,1.5) {$x^{-1}$};


\path[draw] (0,0) -- (1,0) -- (1,.95);
\path[draw] (0,1) -- (0,0);
\path[draw] (0,0) -- (-1,0) -- (-1,1) -- (0,1);
\path[draw] (0,1) -- (0,2) -- (-1,2) -- (-1,1);
\path[draw] (0,1) -- (1,1.05) -- (1,2) -- (0,2) -- (0,1);

\path[draw] (2.5,1) -- (2.5,2) -- (1.5,2) -- (1.5,1);

\draw[dashed, ultra thick, greenline] (1,.95) -- (0,1);
\draw[dashed, ultra thick, greenline] (2.5,1) -- (1.5,1);

\fill[darkblue] (0,1) circle (0.1);
\fill[darkblue] (1.5,1) circle (0.1);



\end{tikzpicture}

%% file: tikz/Beisp-G2_3-1-Zyl.tex
\begin{tikzpicture}[remember picture,scale=0.9]


\fill[green,opacity=0.4](-1,0)--(3,0)--(3,1)--(-1,1);
\fill[darkblue,opacity=0.3](-1,1)--(0,1)--(0,2)--(-1,2);
\fill[darkblue,opacity=0.3](1,1)--(2,1)--(2,2)--(1,2);
\fill[darkblue,opacity=0.3](0,-1)--(1,-1)--(1,0)--(0,0);
\fill[darkblue,opacity=0.3](2,-1)--(3,-1)--(3,0)--(2,0);

\node[scale=0.9] 
(1) at (-0.5,0.5) {$1$};
\node[scale=0.9] 
($i$) at (0.5,0.5) {$r$};
\node[scale=0.9] 
($-1$) at (1.5,0.5) {$r^2$};
\node[scale=0.9] 
($-i$) at (2.5,0.5) {$r^3$};

\node[scale=0.9] 
(j) at (-0.5,1.5) {$s$};
\node[scale=0.9] 
(-k) at (0.5,-0.5) {$rs$};
\node[scale=0.9] 
($-j$) at (1.5,1.5) {$r^2s$};
\node[scale=0.9] 
($-k$) at (2.5,-0.5) {$r^3s$};


\draw (0,0) -- (1,0) -- (1,1) -- (0,1) -- (0,0);
\draw (1,0) -- (2,0) -- (2,1) -- (1,1);
\draw (2,0) -- (3,0) -- (3,1) -- (2,1);
\draw (0,0) -- (-1,0) -- (-1,1) -- (0,1);

\draw (0,0) -- (0,-1) -- (1,-1) -- (1,0);
\draw (2,0) -- (2,-1) -- (3,-1) -- (3,0);

\draw (1,1) -- (1,2) -- (2,2) -- (2,1);
\draw (0,1) -- (0,2) -- (-1,2) -- (-1,1);

\fill[darkblue] (0,0) circle (0.1);
\fill[darkblue] (2,0) circle (0.1);
\fill[darkblue] (2,2) circle (0.1);
\fill[darkblue] (0,2) circle (0.1);

\fill[orangedot] (1,0) circle (0.1);
\fill[orangedot] (3,0) circle (0.1);
\fill[orangedot] (1,2) circle (0.1);
\fill[orangedot] (-1,2) circle (0.1);
\fill[orangedot] (-1,0) circle (0.1);

\fill[gray] (0,1) circle (0.1);
\fill[gray] (2,1) circle (0.1);
\fill[gray] (0,-1) circle (0.1);
\fill[gray] (2,-1) circle (0.1);

\fill[red] (1,-1) circle (0.1);
\fill[red] (3,-1) circle (0.1);
\fill[red] (1,1) circle (0.1);
\fill[red] (3,1) circle (0.1);
\fill[red] (-1,1) circle (0.1);


\draw (-1.1,0.5) -- (-0.9,0.5);
\draw (2.9,0.5) -- (3.1,0.5);

\draw (1.9,-0.45) -- (2.1,-0.45);
\draw (1.9,-0.55) -- (2.1,-0.55);

\draw (-.1,1.45) -- (.1,1.45);
\draw (-.1,1.55) -- (.1,1.55);

\draw (-0.1,-0.5) -- (.1,-0.5);
\draw (-0.1,-0.4) -- (.1,-0.4);
\draw (-0.1,-0.6) -- (.1,-0.6);
\draw (1.9,1.5) -- (2.1,1.5);
\draw (1.9,1.4) -- (2.1,1.4);
\draw (1.9,1.6) -- (2.1,1.6);

\draw (.9,1.6) -- (1.1,1.5) -- (.9,1.4);

\draw (2.9,-.6) -- (3.1,-.5) -- (2.9,-.4);

\draw (-.9,1.6) -- (-1.1,1.5) -- (-.9,1.4);

\draw (1.1,-.6) -- (.9,-.5) -- (1.1,-.4);


\draw (-0.55,1.9) -- (-0.45,2.1);
\draw (-.45,.1) -- (-.55,-.1);

\draw (0.4,-1.1) -- (.5,-.9);
\draw (0.5,-1.1) -- (.6,-.9);

\draw (.4,.9) -- (.5,1.1);
\draw (.5,.9) -- (.6,1.1);

\draw (1.45,2.1) -- (1.55,1.9);
\draw (1.45,.1) -- (1.55,-.1);

\draw (2.4,1.1) -- (2.5,.9);
\draw (2.5,1.1) -- (2.6,.9);

\draw (2.4,-.9) -- (2.5,-1.1);
\draw (2.5,-.9) -- (2.6,-1.1);
\end{tikzpicture}

%% file: tikz/Beisp-G2_3-1-v2-Zyl.tex
\begin{tikzpicture}[remember picture,scale=0.765]

\fill[green,opacity=0.4](0,4)--(2,4)--(2,5)--(0,5);
\fill[darkblue,opacity=0.4](1,3)--(3,3)--(3,4)--(1,4);
\fill[orange,opacity=0.4](2,2)--(4,2)--(4,3)--(2,3);
\fill[gray,opacity=0.4](3,1)--(5,1)--(5,2)--(3,2);

\node[scale=1.02] 
(1) at (3.5,1.5) {$1$};
\node[scale=1.02] 
($i$) at (4.5,1.5) {$s$};
\node[scale=1.02] 
($-1$) at (3.5,2.5) {$rs$};
\node[scale=1.02] 
($-i$) at (2.5,2.5) {$r$};

\node[scale=1.02] 
(j) at (2.5,3.5) {$r^2s$};
\node[scale=1.02] 
(-k) at (1.5,3.5) {$r^2$};
\node[scale=1.02] 
($-j$) at (1.5,4.5) {$r^3s$};
\node[scale=1.02] 
($-k$) at (.5,4.5) {$r^3$};


\draw (0,5) -- (1,5) -- (1,4) -- (0,4) -- (0,5);
\draw (1,5) -- (2,5) -- (2,4) -- (1,4);
\draw (1,4) -- (1,3) -- (2,3) -- (2,4);
\draw (2,4) -- (3,4) -- (3,3) -- (2,3);

\draw (3,3) -- (3,2) -- (2,2) -- (2,3);
\draw (3,3) -- (4,3) -- (4,2) -- (3,2);

\draw (3,2) -- (3,1) -- (4,1) -- (4,2);
\draw (4,2) -- (5,2) -- (5,1) -- (4,1);

\fill[darkblue] (3,1) circle (0.1);
\fill[darkblue] (1,5) circle (0.1);
\fill[darkblue] (1,3) circle (0.1);
\fill[darkblue] (3,3) circle (0.1);
\fill[darkblue] (5,1) circle (0.1);

\fill[orangedot] (0,5) circle (0.1);
\fill[orangedot] (2,5) circle (0.1);
\fill[orangedot] (2,3) circle (0.1);
\fill[orangedot] (4,3) circle (0.1);
\fill[orangedot] (4,1) circle (0.1);

\fill[gray] (1,4) circle (0.1);
\fill[gray] (3,4) circle (0.1);
\fill[gray] (3,2) circle (0.1);
\fill[gray] (5,2) circle (0.1);

\fill[red] (0,4) circle (0.1);
\fill[red] (2,4) circle (0.1);
\fill[red] (2,2) circle (0.1);
\fill[red] (4,2) circle (0.1);


\draw (.5,4.9) -- (.5,5.1);
\draw (4.5,.9) -- (4.5,1.1);

\draw (4.45,1.9) -- (4.45,2.1);
\draw (4.55,1.9) -- (4.55,2.1);

\draw (.55,3.9) -- (.55,4.1);
\draw (.45,3.9) -- (.45,4.1);















\end{tikzpicture}

%% file: tikz/shearing.tex
\begin{tikzpicture}[remember picture,line width=1pt, scale=1]

\newcommand{\vertex}{\node[circle, draw, inner sep=0pt, minimum size=6pt]}


\draw (0,0) -- (1,0) -- (1,1) -- (0,1) -- (0,0);
\draw (0,0) -- (-1,0) -- (-1,1) -- (0,1);

\draw (0,1) -- (0,2) -- (-1,2) -- (-1,1);

\draw (1.5,.5) -- (2.5,0.5);
\draw (2.3,.6) -- (2.5,0.5) -- (2.3,0.4);

\node[scale=1] ($A$) at (2,0.8) {$T^2$};
\node[scale=1] ($=$) at (8.5,0.5) {$=$};






\draw (3,0) -- (4,0) -- (6,1) -- (5,1) -- (3,0);
\draw (5,1) -- (7,2) -- (8,2) -- (6,1);

\draw (4,0) -- (5,0) -- (7,1) -- (6,1);

\draw[dashed, ultra thick, orangeline] (4,0) -- (4,.5);
\draw[dashed, ultra thick, orangeline] (6,1) -- (6,.5);
\draw[dashed, ultra thick, redline] (6,1) -- (6,1.5);
\draw[dashed, ultra thick, redline] (7,2) -- (7,1.5);
\draw[dashed, ultra thick, greenline] (5,1) -- (5,0);


\draw (10,0) -- (11,0);
\draw (11,1) -- (10,1);
\draw (9,1) -- (10,1);
\draw (10,0) -- (9,0);
\draw (10,2) -- (9,2);

\path[draw, ultra thick, greenline] (9,0) -- (9,1);
\path[draw, ultra thick, greenline] (11,0) -- (11,1);
\path[draw, ultra thick, orangeline] (10,0) -- (10,1);
\path[draw, ultra thick, redline] (10,2) -- (10,1);
\path[draw, ultra thick, redline] (9,2) -- (9,1);





\end{tikzpicture}

%% file: tikz/Beisp-G3_3-1-1.tex
\begin{tikzpicture}[remember picture, scale=0.9]

\node[scale=0.9] 
(1) at (-0.5,0.5) {$1$};
\node[scale=0.9] 
($r$) at (0.5,0.5) {$r$};
\node[scale=0.9] 
($r^2$) at (1.5,0.5) {$r^2$};
\node[scale=0.9] 
($r^3$) at (2.5,0.5) {$r^3$};
\node[scale=0.9] 
($r^4$) at (3.5,0.5) {$r^4$};
\node[scale=0.9] 
($r^5$) at (4.5,0.5) {$r^5$};
\node[scale=0.9] 
($r^6$) at (5.5,0.5) {$r^6$};
\node[scale=0.9] 
($r^7$) at (6.5,0.5) {$r^7$};
\node[scale=0.9] 
($r^8$) at (7.5,0.5) {$r^8$};

\node[scale=0.9] 
(s) at (-0.5,1.5) {$s$};
\node[scale=0.9] 
($rs^2$) at (0.5,-0.5) {$rs^2$};
\node[scale=0.9] 
($r^2s$) at (2.5,-0.5) {$r^3s^2$};
\node[scale=0.9] 
($r^3s^2$) at (1.5,1.5) {$r^2s$};
\node[scale=0.9] 
($r^5s^2$) at (4.5,-0.5) {$r^5s^2$};
\node[scale=0.9] 
($r^4s$) at (3.5,1.5) {$r^4s$};
\node[scale=0.9] 
($r^6s$) at (5.5,1.5) {$r^6s$};
\node[scale=0.9] 
($r^7s^2$) at (6.5,-0.5) {$r^7s^2$};
\node[scale=0.9] 
($r^8s^2$) at (7.5,2.5) {$r^8s^2$};

\node[scale=0.9] 
($s^2$) at (-0.5,2.5) {$s^2$};
\node[scale=0.9] 
($rs$) at (0.5,-1.5) {$rs$};
\node[scale=0.9] 
($r^2s^2$) at (2.5,-1.5) {$r^3s$};
\node[scale=0.9] 
($r^3s$) at (1.5,2.5) {$r^2s^2$};
\node[scale=0.9] 
($r^5s$) at (4.5,-1.5) {$r^5s$};
\node[scale=0.9] 
($r^4s^2$) at (3.5,2.5) {$r^4s^2$};
\node[scale=0.9] 
($r^6s^2$) at (5.5,2.5) {$r^6s^2$};
\node[scale=0.9] 
($r^7s$) at (6.5,-1.5) {$r^7s$};
\node[scale=0.9] 
($r^8s$) at (7.5,1.5) {$r^8s$};


\draw (0,0) -- (1,0) -- (1,1) -- (0,1) -- (0,0);
\draw (1,0) -- (2,0) -- (2,1) -- (1,1);
\draw (2,0) -- (3,0) -- (3,1) -- (2,1);
\draw (0,0) -- (-1,0) -- (-1,1) -- (0,1);
\draw (3,0) -- (4,0) -- (4,1) -- (3,1);
\draw (4,0) -- (5,0) -- (5,1) -- (4,1);
\draw (5,0) -- (6,0) -- (6,1) -- (5,1);
\draw (6,0) -- (7,0) -- (7,1) -- (6,1);
\draw (7,0) -- (8,0) -- (8,1) -- (7,1);

\draw (0,0) -- (0,-1) -- (1,-1) -- (1,0);
\draw (2,0) -- (2,-1) -- (3,-1) -- (3,0);
\draw (4,0) -- (4,-1) -- (5,-1) -- (5,0);
\draw (6,0) -- (6,-1) -- (7,-1) -- (7,0);

\draw (0,-1) -- (0,-2) -- (1,-2) -- (1,-1);
\draw (2,-1) -- (2,-2) -- (3,-2) -- (3,-1);
\draw (4,-1) -- (4,-2) -- (5,-2) -- (5,-1);
\draw (6,-1) -- (6,-2) -- (7,-2) -- (7,-1);

\draw (1,1) -- (1,2) -- (2,2) -- (2,1);
\draw (0,1) -- (0,2) -- (-1,2) -- (-1,1);
\draw (4,1) -- (4,2) -- (3,2) -- (3,1);
\draw (6,1) -- (6,2) -- (5,2) -- (5,1);
\draw (8,1) -- (8,2) -- (7,2) -- (7,1);

\draw (1,2) -- (1,3) -- (2,3) -- (2,2);
\draw (0,2) -- (0,3) -- (-1,3) -- (-1,2);
\draw (4,2) -- (4,3) -- (3,3) -- (3,2);
\draw (6,2) -- (6,3) -- (5,3) -- (5,2);
\draw (8,2) -- (8,3) -- (7,3) -- (7,2);

\fill[darkblue] (-1,1) circle (0.1);
\fill[darkblue] (2,1) circle (0.1);
\fill[darkblue] (2,-2) circle (0.1);
\fill[darkblue] (5,1) circle (0.1);
\fill[darkblue] (5,-2) circle (0.1);
\fill[darkblue] (8,1) circle (0.1);

\fill[orangedot] (0,1) circle (0.1);
\fill[orangedot] (3,1) circle (0.1);
\fill[orangedot] (6,1) circle (0.1);
\fill[orangedot] (0,-2) circle (0.1);
\fill[orangedot] (3,-2) circle (0.1);
\fill[orangedot] (6,-2) circle (0.1);

\fill[green] (1,1) circle (0.1);
\fill[green] (4,1) circle (0.1);
\fill[green] (7,1) circle (0.1);
\fill[green] (1,-2) circle (0.1);
\fill[green] (4,-2) circle (0.1);
\fill[green] (7,-2) circle (0.1);

\fill[gray2] (0,0) circle (0.1);
\fill[gray2] (0,3) circle (0.1);
\fill[gray2] (3,0) circle (0.1);
\fill[gray2] (3,3) circle (0.1);
\fill[gray2] (6,0) circle (0.1);
\fill[gray2] (6,3) circle (0.1);

\fill[lightblue2] (1,0) circle (0.1);
\fill[lightblue2] (1,3) circle (0.1);
\fill[lightblue2] (4,0) circle (0.1);
\fill[lightblue2] (4,3) circle (0.1);
\fill[lightblue2] (7,0) circle (0.1);
\fill[lightblue2] (7,3) circle (0.1);

\fill[black] (0,2) circle (0.1);
\fill[black] (3,2) circle (0.1);
\fill[black] (0,-1) circle (0.1);
\fill[black] (3,-1) circle (0.1);
\fill[black] (6,2) circle (0.1);
\fill[black] (6,-1) circle (0.1);

\fill[lila] (-1,2) circle (0.1);
\fill[lila] (2,2) circle (0.1);
\fill[lila] (8,2) circle (0.1);
\fill[lila] (2,-1) circle (0.1);
\fill[lila] (5,2) circle (0.1);
\fill[lila] (5,-1) circle (0.1);

\fill[petrol] (1,2) circle (0.1);
\fill[petrol] (4,2) circle (0.1);
\fill[petrol] (1,-1) circle (0.1);
\fill[petrol] (4,-1) circle (0.1);
\fill[petrol] (7,2) circle (0.1);
\fill[petrol] (7,-1) circle (0.1);

\fill[red] (-1,0) circle (0.1);
\fill[red] (-1,3) circle (0.1);
\fill[red] (2,0) circle (0.1);
\fill[red] (2,3) circle (0.1);
\fill[red] (5,0) circle (0.1);
\fill[red] (5,3) circle (0.1);
\fill[red] (8,0) circle (0.1);
\fill[red] (8,3) circle (0.1);


\draw (-1.1,0.5) -- (-0.9,0.5);
\draw (7.9,0.5) -- (8.1,0.5);

\draw (.1,1.45) -- (-0.1,1.45);
\draw (.1,1.55) -- (-0.1,1.55);

\draw (6.1,-1.45) -- (5.9,-1.45);
\draw (6.1,-1.55) -- (5.9,-1.55);

\draw (7.1,-1.45) -- (6.9,-1.45);
\draw (7.1,-1.5) -- (6.9,-1.5);
\draw (7.1,-1.55) -- (6.9,-1.55);

\draw (4.1,-1.45) -- (3.9,-1.45);
\draw (4.1,-1.5) -- (3.9,-1.5);
\draw (4.1,-1.55) -- (3.9,-1.55);

\draw (5,-1.5) circle (0.05);

\draw (2,-1.5) circle (0.05);

\draw (2.95,-1.4) -- (2.95,-1.5) -- (3.05,-1.5) -- (3.05,-1.4) -- (2.95,-1.4);
\draw (-.05,-1.4) -- (-.05,-1.5) -- (.05,-1.5) -- (.05,-1.4) -- (-0.05,-1.4);

\draw (1.1,-1.4) -- (.9,-1.5);
\draw (7.1,1.5) -- (6.9,1.4);

\draw (8.1,1.55) -- (7.9,1.45);
\draw (8.1,1.65) -- (7.9,1.55);

\draw (5.1,1.55) -- (4.9,1.45);
\draw (5.1,1.65) -- (4.9,1.55);

\draw (6.1,1.55) -- (5.9,1.45);
\draw (6.1,1.6) -- (5.9,1.5);
\draw (6.1,1.65) -- (5.9,1.55);

\draw (3.1,1.55) -- (2.9,1.45);
\draw (3.1,1.6) -- (2.9,1.5);
\draw (3.1,1.65) -- (2.9,1.55);

\draw (3.9,1.5) -- (4.1,1.4);
\draw (4.1,1.6) -- (3.9,1.5);
\draw (4.1,1.65) -- (3.9,1.55);

\draw (.9,1.5) -- (1.1,1.4);
\draw (1.1,1.6) -- (.9,1.5);
\draw (1.1,1.65) -- (.9,1.55);

\draw (1.9,1.5) -- (2.1,1.4);
\draw (2.1,1.6) -- (1.9,1.5);

\draw (-1.1,1.5) -- (-.9,1.4);
\draw (-.9,1.6) -- (-1.1,1.5);

\draw (.1,2.45) -- (-0.1,2.55);

\draw (3.1,2.45) -- (2.9,2.55);

\draw (4.1,2.4) -- (3.9,2.5);
\draw (4.1,2.5) -- (3.9,2.6);

\draw (7.1,2.4) -- (6.9,2.5);
\draw (7.1,2.5) -- (6.9,2.6);

\draw (8.1,2.4) -- (7.9,2.5);
\draw (8.1,2.45) -- (7.9,2.55);
\draw (8.1,2.5) -- (7.9,2.6);

\draw (2.1,-.6) -- (1.9,-.5);
\draw (2.1,-.55) -- (1.9,-.45);
\draw (2.1,-.5) -- (1.9,-.4);

\draw (3.1,-.6) -- (2.9,-.5);
\draw (3.1,-.5) -- (2.9,-.5);
\draw (3.1,-.5) -- (2.9,-.4);

\draw (6.1,-.6) -- (5.9,-.5);
\draw (6.1,-.5) -- (5.9,-.5);
\draw (6.1,-.5) -- (5.9,-.4);

\draw (7.1,-.55) -- (6.9,-.55);
\draw (7.1,-.55) -- (6.9,-.45);
\draw (7.1,-.45) -- (6.9,-.45);

\draw (1.1,2.45) -- (.9,2.45);
\draw (1.1,2.45) -- (.9,2.55);
\draw (1.1,2.55) -- (.9,2.55);

\draw (2.1,2.45) -- (1.9,2.45);
\draw (2.1,2.55) -- (1.9,2.45);
\draw (2.1,2.55) -- (1.9,2.55);

\draw (5.1,2.45) -- (4.9,2.45);
\draw (5.1,2.55) -- (4.9,2.45);
\draw (5.1,2.55) -- (4.9,2.55);

\draw (6.1,2.45) -- (5.9,2.45);
\draw (6.1,2.55) -- (5.9,2.45);
\draw (6.1,2.45) -- (5.9,2.55);
\draw (6.1,2.55) -- (5.9,2.55);

\draw (.1,-.55) -- (-.1,-.55);
\draw (.1,-.45) -- (-.1,-.55);
\draw (.1,-.55) -- (-.1,-.45);
\draw (.1,-.45) -- (-.1,-.45);

\fill[black] (1,-.5) circle (0.05);
\fill[black] (-1,2.5) circle (0.05);








\draw (-0.45,.1) -- (-0.55,-.1);
\draw (-.45,3.1) -- (-.55,2.9);

\draw (1.5,.1) -- (1.4,-.1);
\draw (1.6,.1) -- (1.5,-.1);
\draw (1.5,3.1) -- (1.4,2.9);
\draw (1.6,3.1) -- (1.5,2.9);

\draw (3.5,.1) -- (3.4,-.1);
\draw (3.6,.1) -- (3.5,-.1);
\draw (3.55,.1) -- (3.45,-.1);
\draw (3.5,3.1) -- (3.4,2.9);
\draw (3.6,3.1) -- (3.5,2.9);
\draw (3.55,3.1) -- (3.45,2.9);

\draw (5.6,-.1) -- (5.5,.1) -- (5.4,-.1);
\draw (5.6,2.9) -- (5.5,3.1) -- (5.4,2.9);

\draw (7.6,.1) -- (7.5,-.1) -- (7.4,.1);
\draw (7.6,3.1) -- (7.5,2.9) -- (7.4,3.1);

\draw (0.45,1.1) -- (0.55,.9);
\draw (.45,-1.9) -- (.55,-2.1);

\draw (2.4,-1.9) -- (2.5,-2.1);
\draw (2.5,-1.9) -- (2.6,-2.1);

\draw (2.4,1.1) -- (2.5,.9);
\draw (2.5,1.1) -- (2.6,.9);

\draw (4.4,1.1) -- (4.5,.9);
\draw (4.45,1.1) -- (4.55,.9);
\draw (4.5,1.1) -- (4.6,.9);

\draw (4.4,-1.9) -- (4.5,-2.1);
\draw (4.45,-1.9) -- (4.55,-2.1);
\draw (4.5,-1.9) -- (4.6,-2.1);

\draw (6.45,-1.9) -- (6.55,-2.1);
\draw (6.55,-1.9) -- (6.45,-2.1);
\draw (6.45,1.1) -- (6.55,.9);
\draw (6.55,1.1) -- (6.45,.9);
\end{tikzpicture}

%% file: tikz/staircase-origami.tex
\begin{tikzpicture}[remember picture,scale=0.85]

\fill[darkblue,opacity=0.4](0,0)--(2,0)--(2,1)--(0,1);
\fill[gray,opacity=0.4](1,1)--(3,1)--(3,2)--(1,2);
\fill[green,opacity=0.4](2,2)--(4,2)--(4,3)--(2,3);
\fill[orange,opacity=0.4](3,3)--(5,3)--(5,4)--(3,4);

\node[scale=1.02] 
(1) at (0.5,0.5) {$srs$};
\node[scale=1.02] 
($i$) at (1.5,0.5) {$sr$};
\node[scale=1.02] 
($-1$) at (1.5,1.5) {$1$};
\node[scale=1.02] 
($-i$) at (2.5,1.5) {$s$};

\node[scale=1.02] 
(j) at (2.5,2.5) {$r$};
\node[scale=1.02] 
(-k) at (3.5,2.5) {$rs$};
\node[scale=1.02] 
($-j$) at (3.5,3.5) {$r^2$};
\node[scale=1.02] 
($-k$) at (4.5,3.5) {$r^2s$};
\node[scale=1.02] 
($-k$) at (4.5,4.7) {\reflectbox{$\ddots$}};
\node[scale=1.02] 
($-k$) at (.5,-.5) {\reflectbox{$\ddots$}};

\path[draw, ultra thick, whiteline] (1,0) -- (0,0);
\path[draw, ultra thick, greenline] (1,0) -- (0,0);

\draw (0,0) -- (0,1) -- (1,1) -- (1,0);
\draw (1,0) -- (1,1) -- (2,1) -- (2,0) -- (1,0);

\draw (1,1) -- (1,2) -- (2,2) -- (2,1) -- (1,1);
\draw (2,1) -- (2,2) -- (3,2) -- (3,1) -- (2,1);

\draw (2,2) -- (2,3) -- (3,3) -- (3,2) -- (2,2);
\draw (3,2) -- (3,3) -- (4,3) -- (4,2) -- (3,2);

\draw (3,3) -- (3,4) -- (4,4) -- (4,3) -- (3,3);
\draw (4,3) -- (4,4);
\draw (5,4) -- (5,3) -- (4,3);

\path[draw, whiteline] (4,4) -- (5,4);
\path[draw, ultra thick, greenline] (4,4) -- (5,4);

\fill[darkblue] (0,0) circle (0.1);
\fill[darkblue] (2,0) circle (0.1);
\fill[darkblue] (2,2) circle (0.1);
\fill[darkblue] (4,2) circle (0.1);
\fill[darkblue] (4,4) circle (0.1);

\fill[orangedot] (1,0) circle (0.1);
\fill[orangedot] (1,2) circle (0.1);
\fill[orangedot] (3,2) circle (0.1);
\fill[orangedot] (3,4) circle (0.1);
\fill[orangedot] (5,4) circle (0.1);

\fill[gray] (0,1) circle (0.1);
\fill[gray] (2,1) circle (0.1);
\fill[gray] (2,3) circle (0.1);
\fill[gray] (4,3) circle (0.1);

\fill[red] (1,1) circle (0.1);
\fill[red] (3,1) circle (0.1);
\fill[red] (3,3) circle (0.1);
\fill[red] (5,3) circle (0.1);

\end{tikzpicture}

%% file: tikz/ori2-D_infty.tex
\begin{tikzpicture}[remember picture,scale=0.85]

\fill[gray,opacity=0.4](-2,0)--(5,0)--(5,1)--(-2,1);

\fill[orange,opacity=0.4](-1,1) -- (-1,2) -- (0,2) -- (0,1) -- (-1,1);
\fill[orange,opacity=0.4](1,1) -- (1,2) -- (2,2) -- (2,1) -- (1,1);
\fill[orange,opacity=0.4](3,1) -- (3,2) -- (4,2) -- (4,1) -- (3,1);
\fill[orange,opacity=0.4](-2,-1) -- (-2,0) -- (-1,0) -- (-1,-1) -- (-2,-1);
\fill[orange,opacity=0.4](0,-1) -- (0,0) -- (1,0) -- (1,-1) -- (0,-1);
\fill[orange,opacity=0.4](2,-1) -- (2,0) -- (3,0) -- (3,-1) -- (2,-1);
\fill[orange,opacity=0.4](4,-1) -- (4,0) -- (5,0) -- (5,-1) -- (4,-1);

\node[scale=1.02] 
(1) at (-1.5,0.5) {$r^{-3}$};
\node[scale=1.02] 
(1) at (-0.5,0.5) {$r^{-2}$};
\node[scale=1.02] 
(1) at (0.5,0.5) {$r^{-1}$};
\node[scale=1.02] 
($i$) at (1.5,0.5) {$1$};
\node[scale=1.02] 
($i$) at (2.5,0.5) {$r$};
\node[scale=1.02] 
($i$) at (3.5,0.5) {$r^2$};
\node[scale=1.02] 
($i$) at (4.5,0.5) {$r^3$};

\node[scale=.92] 
($-1$) at (-.5,1.5) {$r^{-2}s$};
\node[scale=1.02] 
($-1$) at (1.5,1.5) {$s$};
\node[scale=1.02] 
($-1$) at (3.5,1.5) {$r^2s$};

\node[scale=.9] 
($-1$) at (-1.5,-.5) {$r^{-3}s$};
\node[scale=.9] 
($-1$) at (.5,-.5) {$r^{-1}s$};
\node[scale=1.02] 
($-1$) at (2.5,-.5) {$rs$};
\node[scale=1.02] 
($-1$) at (4.5,-.5) {$r^3s$};

\node[scale=1.02] 
($-k$) at (-2.7,.5) {$\dots$};
\node[scale=1.02] 
($-k$) at (5.7,.5) {$\dots$};

\path[draw, whiteline] (5,0) -- (5,1);
\path[draw, ultra thick, greenline] (-2,0) -- (-2,1);
\path[draw, ultra thick, greenline] (5,0) -- (5,1);

\draw (-2,1) -- (-1,1) -- (-1,0) -- (-2,0);
\draw (-1,0) -- (-1,1) -- (0,1) -- (0,0) -- (-1,0);
\draw (0,0) -- (0,1) -- (1,1) -- (1,0) -- (0,0);
\draw (1,0) -- (1,1) -- (2,1) -- (2,0) -- (1,0);
\draw (2,0) -- (2,1) -- (3,1) -- (3,0) -- (2,0);
\draw (3,0) -- (3,1) -- (4,1) -- (4,0) -- (3,0);
\draw (5,0) -- (4,0) -- (4,1) -- (5,1);

\draw (-1,1) -- (-1,2) -- (0,2) -- (0,1) -- (-1,1);
\draw (1,1) -- (1,2) -- (2,2) -- (2,1) -- (1,1);
\draw (3,1) -- (3,2) -- (4,2) -- (4,1) -- (3,1);

\draw (-2,-1) -- (-2,0) -- (-1,0) -- (-1,-1) -- (-2,-1);
\draw (0,-1) -- (0,0) -- (1,0) -- (1,-1) -- (0,-1);
\draw (2,-1) -- (2,0) -- (3,0) -- (3,-1) -- (2,-1);
\draw (4,-1) -- (4,0) -- (5,0) -- (5,-1) -- (4,-1);

\fill[darkblue] (-2,0) circle (0.1);
\fill[darkblue] (0,0) circle (0.1);
\fill[darkblue] (0,2) circle (0.1);
\fill[darkblue] (2,0) circle (0.1);
\fill[darkblue] (2,2) circle (0.1);
\fill[darkblue] (4,2) circle (0.1);
\fill[darkblue] (4,0) circle (0.1);

\fill[orangedot] (-1,0) circle (0.1);
\fill[orangedot] (-1,2) circle (0.1);
\fill[orangedot] (1,0) circle (0.1);
\fill[orangedot] (1,2) circle (0.1);
\fill[orangedot] (3,2) circle (0.1);
\fill[orangedot] (3,0) circle (0.1);
\fill[orangedot] (5,0) circle (0.1);

\fill[gray] (-2,1) circle (0.1);
\fill[gray] (-2,-1) circle (0.1);
\fill[gray] (0,1) circle (0.1);
\fill[gray] (0,-1) circle (0.1);
\fill[gray] (2,1) circle (0.1);
\fill[gray] (2,-1) circle (0.1);
\fill[gray] (4,1) circle (0.1);
\fill[gray] (4,-1) circle (0.1);

\fill[red] (-1,-1) circle (0.1);
\fill[red] (-1,1) circle (0.1);
\fill[red] (1,-1) circle (0.1);
\fill[red] (1,1) circle (0.1);
\fill[red] (3,1) circle (0.1);
\fill[red] (3,-1) circle (0.1);
\fill[red] (5,1) circle (0.1);
\fill[red] (5,-1) circle (0.1);


\draw (1.9,1.5) -- (2.1,1.5);
\draw (-.1,-0.5) -- (0.1,-0.5);

\draw (1.1,1.45) -- (0.9,1.45);
\draw (1.1,1.55) -- (0.9,1.55);
\draw (3.1,-.45) -- (2.9,-.45);
\draw (3.1,-.55) -- (2.9,-.55);

\draw (2.1,-.45) -- (1.9,-.45);
\draw (2.1,-.5) -- (1.9,-.5);
\draw (2.1,-.55) -- (1.9,-.55);
\draw (4.1,1.45) -- (3.9,1.45);
\draw (4.1,1.5) -- (3.9,1.5);
\draw (4.1,1.55) -- (3.9,1.55);

\draw (5.1,-.55) -- (4.9,-.45);
\draw (3.1,1.45) -- (2.9,1.55);

\draw (4.1,-.5) -- (3.9,-.4);
\draw (4.1,-.6) -- (3.9,-.5);

\draw (1.1,-.45) -- (.9,-.55);
\draw (-1.1,1.45) -- (-.9,1.55);

\draw (-1.9,-.4) -- (-2.1,-.5);
\draw (-1.9,-.5) -- (-2.1,-.6);
\draw (-.1,1.4) -- (.1,1.5);
\draw (-.1,1.5) -- (.1,1.6);

\draw (-.9,-.4) -- (-1.1,-.5);
\draw (-.9,-.45) -- (-1.1,-.55);
\draw (-.9,-.5) -- (-1.1,-.6);
\end{tikzpicture}

%% file: tikz/generalized-Wollmilchsau.tex
\begin{tikzpicture}[remember picture, scale=0.9]

\node[scale=0.9] 
(1) at (-0.5,0.5) {$1$};
\node[scale=0.9] 
($x$) at (0.5,0.5) {$x$};
\node[scale=0.9] 
($x^2$) at (1.5,0.5) {$x^2$};
\node[scale=0.9] 
($x^3$) at (2.5,0.5) {$x^3$};
\node[scale=0.68] 
($x^4$) at (3.5,0.5) {$x^4=y^4$};
\node[scale=0.9] 
($x^5$) at (4.5,0.5) {$x^5$};
\node[scale=0.9] 
($x^6$) at (5.5,0.5) {$x^6$};
\node[scale=0.9] 
($x^7$) at (6.5,0.5) {$x^7$};

\node[scale=0.9] 
(y) at (-0.5,1.5) {$y$};
\node[scale=0.9] 
($xy^7$) at (0.5,-0.5) {$xy^7$};
\node[scale=0.9] 
($x^3y^7$) at (2.5,-0.5) {$x^3y^7$};
\node[scale=0.9] 
($x^3y^2$) at (1.5,1.5) {$x^2y$};
\node[scale=0.9] 
($xy^3$) at (4.5,-0.5) {$xy^3$};
\node[scale=0.9] 
($y^5$) at (3.5,1.5) {$y^5$};
\node[scale=0.9] 
($x^2y^5$) at (5.5,1.5) {$x^2y^5$};
\node[scale=0.9] 
($x^3y^3$) at (6.5,-0.5) {$x^3y^3$};

\node[scale=0.9] 
($y^2$) at (-0.5,2.5) {$y^2$};
\node[scale=0.9] 
($y^3$) at (-0.5,3.5) {$y^3$};
\node[scale=0.9] 
($xy^6$) at (0.5,-1.5) {$xy^6$};
\node[scale=0.9] 
($xy^5$) at (0.5,-2.5) {$xy^5$};
\node[scale=0.9] 
($x^3y^6$) at (2.5,-1.5) {$x^3y^6$};
\node[scale=0.9] 
($x^3y^5$) at (2.5,-2.5) {$x^3y^5$};
\node[scale=0.9] 
($x^2y^2$) at (1.5,2.5) {$x^2y^2$};
\node[scale=0.9] 
($x^2y^3$) at (1.5,3.5) {$x^2y^3$};
\node[scale=0.9] 
($xy^2$) at (4.5,-1.5) {$xy^2$};
\node[scale=0.9] 
($xy$) at (4.5,-2.5) {$xy$};
\node[scale=0.9] 
($y^6$) at (3.5,2.5) {$y^6$};
\node[scale=0.9] 
($y^7$) at (3.5,3.5) {$y^7$};
\node[scale=0.9] 
($x^2y^6$) at (5.5,2.5) {$x^2y^6$};
\node[scale=0.9] 
($x^2y^7$) at (5.5,3.5) {$x^2y^7$};
\node[scale=0.9] 
($x^3y^2$) at (6.5,-1.5) {$x^3y^2$};
\node[scale=0.9] 
($x^3y$) at (6.5,-2.5) {$x^3y$};


\draw (0,0) -- (1,0) -- (1,1) -- (0,1) -- (0,0);
\draw (1,0) -- (2,0) -- (2,1) -- (1,1);
\draw (2,0) -- (3,0) -- (3,1) -- (2,1);
\draw (0,0) -- (-1,0) -- (-1,1) -- (0,1);
\draw (3,0) -- (4,0) -- (4,1) -- (3,1);
\draw (4,0) -- (5,0) -- (5,1) -- (4,1);
\draw (5,0) -- (6,0) -- (6,1) -- (5,1);
\draw (6,0) -- (7,0) -- (7,1) -- (6,1);

\draw (0,0) -- (0,-1) -- (1,-1) -- (1,0);
\draw (2,0) -- (2,-1) -- (3,-1) -- (3,0);
\draw (4,0) -- (4,-1) -- (5,-1) -- (5,0);
\draw (6,0) -- (6,-1) -- (7,-1) -- (7,0);

\draw (0,-1) -- (0,-2) -- (1,-2) -- (1,-1);
\draw (2,-1) -- (2,-2) -- (3,-2) -- (3,-1);
\draw (4,-1) -- (4,-2) -- (5,-2) -- (5,-1);
\draw (6,-1) -- (6,-2) -- (7,-2) -- (7,-1);

\draw (0,-2) -- (0,-3) -- (1,-3) -- (1,-2);
\draw (2,-2) -- (2,-3) -- (3,-3) -- (3,-2);
\draw (4,-2) -- (4,-3) -- (5,-3) -- (5,-2);
\draw (6,-2) -- (6,-3) -- (7,-3) -- (7,-2);

\draw (1,1) -- (1,2) -- (2,2) -- (2,1);
\draw (0,1) -- (0,2) -- (-1,2) -- (-1,1);
\draw (4,1) -- (4,2) -- (3,2) -- (3,1);
\draw (6,1) -- (6,2) -- (5,2) -- (5,1);

\draw (1,2) -- (1,3) -- (2,3) -- (2,2);
\draw (0,2) -- (0,3) -- (-1,3) -- (-1,2);
\draw (4,2) -- (4,3) -- (3,3) -- (3,2);
\draw (6,2) -- (6,3) -- (5,3) -- (5,2);

\draw (0,3) -- (0,4) -- (-1,4) -- (-1,3);
\draw (2,3) -- (2,4) -- (1,4) -- (1,3);
\draw (4,3) -- (4,4) -- (3,4) -- (3,3);
\draw (6,3) -- (6,4) -- (5,4) -- (5,3);

%
\fill[darkblue] (-1,1) circle (0.1);
\fill[darkblue] (-1,3) circle (0.1);
\fill[darkblue] (3,1) circle (0.1);
\fill[darkblue] (3,-1) circle (0.1);
\fill[darkblue] (3,-3) circle (0.1);
\fill[darkblue] (3,3) circle (0.1);
\fill[darkblue] (7,1) circle (0.1);
\fill[darkblue] (7,-1) circle (0.1);
\fill[darkblue] (7,-3) circle (0.1);

%
\fill[green] (1,3) circle (0.1);
\fill[green] (1,1) circle (0.1);
\fill[green] (1,-1) circle (0.1);
\fill[green] (1,-3) circle (0.1);
\fill[green] (5,3) circle (0.1);
\fill[green] (5,1) circle (0.1);
\fill[green] (5,-1) circle (0.1);
\fill[green] (5,-3) circle (0.1);
\fill[gray2] (0,4) circle (0.1);
\fill[gray2] (0,2) circle (0.1);
\fill[gray2] (0,0) circle (0.1);
\fill[gray2] (0,-2) circle (0.1);
\fill[gray2] (4,4) circle (0.1);
\fill[gray2] (4,2) circle (0.1);
\fill[gray2] (4,0) circle (0.1);
\fill[gray2] (4,-2) circle (0.1);
\fill[lightblue2] (1,4) circle (0.1);
\fill[lightblue2] (1,2) circle (0.1);
\fill[lightblue2] (1,0) circle (0.1);
\fill[lightblue2] (1,-2) circle (0.1);
\fill[lightblue2] (5,4) circle (0.1);
\fill[lightblue2] (5,2) circle (0.1);
\fill[lightblue2] (5,0) circle (0.1);
\fill[lightblue2] (5,-2) circle (0.1);
\fill[black] (0,3) circle (0.1);
\fill[black] (0,1) circle (0.1);
\fill[black] (0,-1) circle (0.1);
\fill[black] (0,-3) circle (0.1);
\fill[black] (4,3) circle (0.1);
\fill[black] (4,1) circle (0.1);
\fill[black] (4,-1) circle (0.1);
\fill[black] (4,-3) circle (0.1);
\fill[lila] (2,4) circle (0.1);
\fill[lila] (2,2) circle (0.1);
\fill[lila] (2,0) circle (0.1);
\fill[lila] (2,-2) circle (0.1);
\fill[lila] (6,4) circle (0.1);
\fill[lila] (6,2) circle (0.1);
\fill[lila] (6,0) circle (0.1);
\fill[lila] (6,-2) circle (0.1);
\fill[petrol] (2,3) circle (0.1);
\fill[petrol] (2,1) circle (0.1);
\fill[petrol] (2,-1) circle (0.1);
\fill[petrol] (2,-3) circle (0.1);
\fill[petrol] (6,3) circle (0.1);
\fill[petrol] (6,1) circle (0.1);
\fill[petrol] (6,-1) circle (0.1);
\fill[petrol] (6,-3) circle (0.1);
\fill[orangedot] (-1,0) circle (0.1);
\fill[orangedot] (-1,2) circle (0.1);
\fill[orangedot] (-1,4) circle (0.1);
\fill[orangedot] (3,-2) circle (0.1);
\fill[orangedot] (3,0) circle (0.1);
\fill[orangedot] (3,2) circle (0.1);
\fill[orangedot] (3,4) circle (0.1);
\fill[orangedot] (7,0) circle (0.1);
\fill[orangedot] (7,-2) circle (0.1);


\def\nextflag#1{%
\pgfmathsetmacro{\test}{isodd(\flags)}%
\ifnum\test=1 \draw #1; \fi%
\pgfmathsetmacro\flags{div(\flags,2)}%
}

\def\mymark#1#2#3{
\begin{scope}[xshift=#2cm,yshift=#3cm]
\pgfmathsetmacro\flags{#1}
\nextflag{(-.1,0) -- (.1,0)} 
\nextflag{(-.1,.1) -- (.1,.1)} 
\nextflag{(-.1,-.1) -- (.1,-.1)} 
\nextflag{(0,0) circle (.1)} 
\nextflag{(-.1,.1) -- (.1,-.1)} 
\nextflag{(-.1,-.1) -- (.1,.1)} 
\nextflag{(-.1,.1) -- (.1,0) -- (-.1,-.1)} 
\nextflag{(.1,.1) -- (-.1,0) -- (.1,-.1)} 
\nextflag{(-.1,-.1) -- (-.1,.1)} 
\nextflag{(.1,-.1) -- (.1,.1)} 
\end{scope}
}
\mymark{1}{-1}{0.5} \mymark{1}{7}{0.5}
\mymark{6}{-1}{1.5} \mymark{6}{7}{-0.5}
\mymark{7}{-1}{2.5} \mymark{7}{7}{-1.5}
\mymark{8}{-1}{3.5} \mymark{8}{7}{-2.5}
\mymark{16}{5}{1.5} \mymark{16}{5}{-0.5}
\mymark{18}{5}{2.5} \mymark{18}{5}{-1.5}
\mymark{20}{5}{3.5} \mymark{20}{5}{-2.5}
\mymark{32}{3}{1.5} \mymark{32}{3}{-0.5}
\mymark{34}{3}{2.5} \mymark{34}{3}{-1.5}
\mymark{36}{3}{3.5} \mymark{36}{3}{-2.5}
\mymark{48}{1}{1.5} \mymark{48}{1}{-0.5}
\mymark{50}{1}{2.5} \mymark{50}{1}{-1.5}
\mymark{52}{1}{3.5} \mymark{52}{1}{-2.5}
\mymark{64}{4}{-2.5} \mymark{64}{4}{3.5}
\mymark{66}{4}{-1.5} \mymark{66}{4}{2.5}
\mymark{68}{4}{-.5} \mymark{68}{4}{1.5}
\mymark{128}{2}{-2.5} \mymark{128}{2}{3.5}
\mymark{130}{2}{-1.5} \mymark{130}{2}{2.5}
\mymark{132}{2}{-.5} \mymark{132}{2}{1.5}
\mymark{257}{0}{-2.5} \mymark{257}{0}{3.5}
\mymark{262}{0}{-1.5} \mymark{262}{0}{2.5}
\mymark{263}{0}{-.5} \mymark{263}{0}{1.5}
\mymark{513}{6}{-2.5} \mymark{513}{6}{3.5}
\mymark{518}{6}{-1.5} \mymark{518}{6}{2.5}
\mymark{519}{6}{-.5} \mymark{519}{6}{1.5}

\mymark{2+4+16+32}{-0.5}{4} \mymark{2+4+16+32}{3.5}{0}
\mymark{768}{1.5}{4} \mymark{768}{5.5}{0}
\mymark{2+4+16+32+256+512}{3.5}{4} \mymark{2+4+16+32+256+512}{-0.5}{0}
\mymark{816}{5.5}{4} \mymark{816}{1.5}{0}

\mymark{64+256}{0.5}{-3} \mymark{64+256}{4.5}{1}
\mymark{64+512}{2.5}{-3} \mymark{64+512}{6.5}{1}
\mymark{128+256}{4.5}{-3} \mymark{128+256}{0.5}{1}
\mymark{128+512}{6.5}{-3} \mymark{128+512}{2.5}{1}

\end{tikzpicture}